\documentclass[a4paper,11pt]{amsart}

\usepackage[T1]{fontenc}
\usepackage[utf8]{inputenc}
\usepackage[margin=3cm]{geometry}

\usepackage{amsmath,amssymb,amsthm,mathtools}
\usepackage{mathrsfs}
\usepackage{enumitem}
\usepackage{tikz-cd}
\usepackage{xcolor}
\usepackage{hyperref}
\usepackage[nameinlink,noabbrev]{cleveref}
\usepackage{float}

\hypersetup{
  colorlinks=true,
  linkcolor=blue!60!black,
  citecolor=blue!60!black,
  urlcolor=blue!60!black
}

\newtheorem{theorem}{Theorem}[section]
\newtheorem{proposition}[theorem]{Proposition}
\newtheorem{lemma}[theorem]{Lemma}
\newtheorem{corollary}[theorem]{Corollary}

\theoremstyle{definition}

\newtheorem{example}[theorem]{Example}

\theoremstyle{remark}
\newtheorem{remark}[theorem]{Remark}

\newcommand{\C}{\mathbb{C}}
\newcommand{\R}{\mathbb{R}}
\newcommand{\Z}{\mathbb{Z}}
\newcommand{\PP}{\mathbb{P}}
\newcommand{\CP}{\mathbb{CP}}
\newcommand{\HP}{\mathbb{HP}}
\newcommand{\HH}{\mathbb{H}}

\newcommand{\SO}{\mathrm{SO}}
\newcommand{\SU}{\mathrm{SU}}
\newcommand{\U}{\mathrm{U}}
\newcommand{\SL}{\mathrm{SL}}

\newcommand{\PSU}{\mathrm{PSU}}
\newcommand{\PGL}{\mathrm{PGL}}

\newcommand{\Qtwotwo}{Q^{2,2}}

\newcommand{\Id}{\mathrm{Id}}
\newcommand{\Span}{\mathrm{Span}}
\newcommand{\diag}{\mathrm{diag}}

\newcommand{\rad}{\mathrm{rad}}

\newcommand{\Ker}{\mathrm{Ker}}

\newcommand{\norm}[1]{\left\lVert #1\right\rVert}

\title[Flat CR Twistor Model]{The Flat CR Twistor Model \(Q^{2,2}\) and Its Algebraic Sections
}

\author{Amedeo Altavilla}
\address{Dipartimento di Matematica, Università degli Studi di Bari Aldo Moro, Italy}
\email{amedeo.altavilla@uniba.it}

\author{Stefano Marini}
\address{Dipartimento di Scienze Matematiche, Fisiche e Informatiche, Università degli Studi di Parma, Italy}
\email{stefano.marini@unipr.it}

\date{}


\begin{document}

\begin{abstract}
We study the flat CR twistor model $Q^{2,2}\subset \mathbb{CP}^3$ by explicit projective methods. Using the anti-holomorphic involution $j$ associated with the twistor fibration, we classify the projective lines contained in $Q^{2,2}$ into twistor fibres and transverse lines, and relate the latter to round $2$-spheres in $S^3$ through an explicit incidence--tangency correspondence. We classify hyperplane sections under the twistor-compatible symmetry group $PSp(1,1)$ and describe the induced CR geometries on $S^3$. For smooth $j$-invariant quadric sections, we obtain a complete relative classification in terms of Coxeter's inversive distance and show that, in the disjoint case, the construction yields an explicit one-parameter family of globally defined real-analytic non-spherical Levi-nondegenerate CR structures on $S^3$.
\end{abstract}

\maketitle

 \setcounter{tocdepth}{1}
 \tableofcontents

\section*{Introduction}

The interplay between projective geometry and twistor theory has proved to be a
remarkably productive source of explicit results in differential geometry. Starting
from Penrose's original construction \cite{Penrose1967,Penrose1976}, the projective
realization of twistor spaces over anti-self-dual conformal $4$-manifolds has made it
possible to translate geometric questions about the base into concrete problems about
curves, and surfaces in $\CP^3$, yielding explicit classifications that would
be hard to obtain by purely differential-geometric methods
\cite{AtiyahHitchinSinger1978,Salamon1982,SalamonViaclovsky2009,Altavilla2025}.

An odd-dimensional analogue of this story is the $CR$ twistor construction
introduced by LeBrun \cite[Section~2]{LeBrun1984}: to every conformal Riemannian $3$-manifold
$(X,[g])$ one associates a $5$-dimensional $CR$ manifold of hypersurface type and
signature of the Levi form $(1,1)$, canonically fibred by Riemann spheres, whose $CR$ geometry
encodes the conformal geometry of the base. In the flat case, i.e. when $(X,[g])$ is the
round $3$-sphere, this construction produces the  hyperquadric
\[
\Qtwotwo
=
\bigl\{[z_0:z_1:z_2:z_3]\in\CP^3\;\big|\;
|z_0|^2+|z_1|^2=|z_2|^2+|z_3|^2\bigr\},
\]
the flat model of Levi-nondegenerate $CR$ geometry of Levi form signature $(1,1)$ in the
sense of Chern--Moser \cite{ChernMoser1974}. This object sits at the crossing of
three theories: twistor geometry, $CR$ geometry, and projective algebraic geometry,
and it is the central object of study of the present paper.

The $CR$ twistor construction has generated a rich geometry to study since LeBrun's
original article. Belgun studied normal $CR$ structures and their relation to the
twistor $CR$ manifold \cite{Belgun2001a,Belgun2001b}. Rossi extended LeBrun's
nonrealizability theory to higher dimensions \cite{Rossi1985} and use it to construct example of non realizable (in sense of \cite{AndreottiFredricks1979}) $CR$ manifold. More recently, Verbitsky \cite{Verbitsky2011} generalized LeBrun's construction
to \(7\)-dimensional \(G_2\)-manifolds, while Fiorenza--V\^an L\^e
\cite{FiorenzaLe2023} studied these constructions in general dimension with the use of
vector cross-product geometry, treating in particular the cases arising from \(G_2\)- and  \(\mathrm{Spin}(7)\)-structures.
Marugame \cite{Marugame2025} gave an interpretation of the construction via the
Fefferman metric, showing that the conformal geometry of the base is completely
encoded in the Fefferman conformal structure of the $CR$ twistor manifold.

On the projective-twistor side, the work of Salamon--Viaclovsky \cite{SalamonViaclovsky2009}
showed that $j$-invariant quadrics in $\CP^3$, the real quadrics in their
terminology, admit a complete classification by their discriminant loci, and that
this classification has rich geometric consequences for orthogonal complex structures
on domains in $\R^4$. More recently, \cite{Altavilla2022,Altavilla2025} developed
systematic projective techniques for twistor geometry in the flag threefold setting,
providing a dictionary between projective objects upstairs and conformal objects
downstairs that applies equally to the $\CP^3$ and flag models.

The purpose of the present article is to apply this projective philosophy
systematically to the flat $CR$ twistor model $\Qtwotwo\subset\CP^3$. The central
observation is that the projective realization of $\Qtwotwo$ makes it possible to
obtain \emph{explicit and complete} answers to several natural geometric questions, about its symmetry group, its holomorphic curves, its hyperplane and quadric
sections (see~\cite{Shafikov2025} and the references therein for a recent report on this topic). Moreover the $CR$ structures of these sections are induced on the base $3$-sphere, e.g. by reducing them to concrete linear and multilinear algebra in $\C^4$.
\par
The main results of the paper are the following.
\par
We show that the projective lines contained in $\Qtwotwo$ are in natural bijection
with the unitary group $U(2)$, and split into two twistorially distinct families:
the $j$-invariant lines (the twistor fibres, parametrized by $SU(2)\cong S^3$) and
the transverse lines (parametrized by $U(2)\setminus SU(2)$, grouped in
$j$-conjugate pairs). Every transverse line $\ell_A$ projects to a round $2$-sphere
in $S^3$, and the incidence geometry of lines in $\Qtwotwo$ translates exactly into
the tangency geometry of round $2$-spheres in $S^3$, two transverse lines meet if
and only if their spheres are tangent with compatible orientation, and a line meets
its $j$-conjugate if and only if the spheres are tangent with opposite orientation.
This gives a completely explicit twistor study of sphere geometry in $S^3$,
in the spirit of Shapiro \cite{Shapiro2013}.

We classify the hyperplane sections of $\Qtwotwo$ under the twistor-compatible
symmetry group $PSp(1,1)$: there are exactly three orbits, determined
by the sign of an invariant related to the restriction on the chosen hyperplane of the Hermitian form defining \(\Qtwotwo\).
 Non-tangent sections are
smooth spherical $CR$ $3$-spheres; tangent sections produce a singular Levi-flat
$CR$ structure on $S^3\setminus\{x\}$. The non-tangent case recovers the standard
spherical $CR$ geometry of $S^3$ via a global diffeomorphism with $\Qtwotwo$.

Lastly, we develop a complete relative classification of smooth $j$-invariant quadrics with
respect to the distinguished $3$-sphere $\Sigma\subset S^4$, using the inversive
distance of Coxeter \cite{Coxeter1966} as the key invariant. The discriminant locus
of each quadric \(\mathcal{S}\) is a geometric circle in $S^4$, and its position relative to $\Sigma$, contained, secant, tangent, or disjoint, determines the branch locus
$\Gamma_{\mathcal{S}}\subset\Sigma$ and hence the global domain of the induced
$CR$ geometry. In the disjoint case, the construction produces, for each
value of the inversive distance, a pair of globally defined conjugate real-analytic
Levi-nondegenerate $CR$ structures on $S^3$ that are generically non-spherical and
non-equivalent to the flat model. This yields an explicit
one-parameter family of non-spherical $CR$ structures on $S^3$, parametrized by a
conformal invariant  and constructed by entirely projective means.

\noindent
\textbf{Organization of the paper.}
Section~\ref{sec:background} recalls the background on twistor spaces, $CR$ manifolds,
LeBrun's construction, and the two ambient spaces $\CP^3$ and $\mathbb{F}$.
Section~\ref{sec:embedded-spheres} compares the two natural embedded models of $S^3$
and their twistor lifts.
Section~\ref{sec:explicit-CR} develops the explicit $CR$ geometry of $\Qtwotwo$
affine charts, local frames, contact form, Levi form.
Section~\ref{sec:symmetries} determines the full projective stabilizer of $\Qtwotwo$
and the twistor-compatible subgroup $PSp(1,1)$.
Section~\ref{sec:curves} classifies the holomorphic curves in $\Qtwotwo$ and
studies the relations between line geometry upstairs and sphere geometry in $S^3$.
Section~\ref{sec:hyperplane-sections} classifies hyperplane sections under $PSp(1,1)$ and
describes the induced $CR$ geometries on the base.
Section~\ref{sec:j-invariant-quadrics} develops the relative classification of smooth
$j$-invariant quadrics, introduces the inversive distance as the key invariant, and
constructs a one-parameter family of $CR$ structures on $S^3$.
\section{Background: twistor geometry, \texorpdfstring{$CR$}{CR} manifolds, and ambient spaces}
\label{sec:background}

This section fixes the geometric framework for the rest of the paper. The flat $CR$
twistor model $\Qtwotwo\subset\CP^3$ sits at the intersection of twistor geometry,
conformal geometry, and Levi--nondegenerate $CR$ geometry, and we briefly recall
the relevant background from each, the classical twistor-space construction for
Riemannian manifolds, the basic language of $CR$ manifolds, and LeBrun's
odd-dimensional twistor construction for conformal $3$-manifolds
\cite{AtiyahHitchinSinger1978,DragomirTomassini2006,LeBrun1984}. This will also
clarify why $\CP^3$ and the full flag threefold are the natural algebraic twistor
ambient spaces here \cite{Hitchin1981,Altavilla2025}.

\subsection{Twistor spaces in Riemannian geometry}
\label{TSRG}
Twistor theory provides a bridge between Riemannian geometry and complex geometry
by replacing geometric data on a real manifold with a space whose fibres parametrize
compatible complex structures
\cite{AtiyahHitchinSinger1978,BarNannicini1995,deBartolomeisNannicini1998}.
In its most classical form, given an oriented
Riemannian manifold \((N,g)\) of even real dimension \(2n\), one considers the bundle
\[
\pi: Z(N)\longrightarrow N
\]
whose fibre over a point \(p\in N\) consists of all endomorphisms
\[
J\in \mathrm{End}(T_pN)
\]
such that
\[
J^2=-\Id,
\qquad
g_p(Ju,Jv)=g_p(u,v)
\quad
\forall\,u,v\in T_pN.
\]
Thus the fibre \(\pi^{-1}(p)\) is the space of orthogonal complex structures on
\(T_pN\) compatible with the metric and the orientation.  Equivalently, \(Z(N)\) is
the bundle associated with the orthonormal frame bundle of \(N\) with typical fibre
\[
\SO(2n)/\U(n).
\]
This is the basic twistor bundle introduced in Riemannian geometry and will serve
as the starting point for the more specific constructions recalled below.

A particularly useful description is available in real dimension four.  In that case
the bundle of \(2\)-forms splits as
\[
\Lambda^2T^*N=\Lambda^+T^*N\oplus\Lambda^-T^*N,
\]
where \(\Lambda^\pm T^*N\) denote the self-dual and anti-self-dual summands.
If \(\omega\in \Lambda^+T_p^*N\) is a unit self-dual \(2\)-form, then \(\omega\)
determines an orthogonal complex structure \(J_\omega\) on \(T_pN\) by
\begin{equation}\label{eq:twistor-correspondence}
g_p(J_\omega u,v)=\omega(u,v),
\qquad u,v\in T_pN.
\end{equation}
In this way one obtains a natural identification
\[
Z(N)\cong S(\Lambda^+T^*N),
\]
where \(S(\Lambda^+T^*N)\) is the unit sphere bundle in \(\Lambda^+T^*N\) with
respect to the metric induced by \(g\)
\cite{BarNannicini1995,Salamon1982}.

\begin{proposition}
[\cite{AtiyahHitchinSinger1978}]
Let \((N,g)\) be an oriented Riemannian \(4\)-manifold. Then the correspondence
\eqref{eq:twistor-correspondence} gives a natural identification
\[
Z_p(N)\cong S(\Lambda^+T_p^*N)\cong S^2\cong \CP^1
\]
for every \(p\in N\).
\end{proposition}

This  picture is one of the main reasons why twistor theory is
especially effective in dimension four,  the fibres become \(2\)-spheres, and the
geometry of the twistor space is closely tied to the self-dual/anti-self-dual
decomposition of curvature.

From the point of view of complex geometry, twistor spaces become especially rigid in
low dimension.  In complex dimension three, and in the algebraic category, the
classical theory singles out two fundamental ambient models
\cite{Hitchin1981}.  
The first is the
projective space \(\CP^3\), which is the twistor space of the round \(4\)-sphere
\((S^4,g_{\mathrm{rnd}})\), identified with the quaternionic projective line
\(\HP^1\).  The twistor projection is the quaternionic  fibration
\begin{equation}\label{eq:Hopf-background}
\pi:\CP^3\longrightarrow \HP^1\simeq S^4,
\qquad
[z_0:z_1:z_2:z_3]\longmapsto [z_0+jz_1:z_2+jz_3],
\end{equation}
whose fibres are projective lines \(\CP^1\subset \CP^3\).
\par
The second is the full
flag threefold. Writing \(p=[p_0:p_1:p_2]\in\CP^2\) and \(\ell=[\ell_0:\ell_1:\ell_2]\in(\CP^2)^*\),
the incidence relation is \(\ell\cdot p:=\ell_0p_0+\ell_1p_1+\ell_2p_2=0\), so that the flag threefold is given by 
\[
\mathbb F=\{(p,\ell)\in \CP^2\times(\CP^2)^*:\ell\cdot p=0\},
\]

and the twistor projection 
\[
\pi_{\mathbb F}:\mathbb{F}
\to\CP^2,
\]
is given by
\begin{equation}\label{eq:flag-twistor-projection}
\pi_{\mathbb F}(p,\ell)=p^*\times\ell=
[\overline p_1\ell_2-\overline p_2\ell_1:
 \overline p_2\ell_0-\overline p_0\ell_2:
 \overline p_0\ell_1-\overline p_1\ell_0]
\end{equation}

This threefold appears as the twistor space over
\(\CP^2\) endowed with its Fubini--Study conformal class
\cite{BurstallRawnsley1990,Altavilla2025}.

These two ambient spaces will play a distinguished role throughout the paper.  On the
one hand, they provide the two natural complex projective settings in which the
embedded models relevant to us arise.  On the other hand, they already exhibit the two
main geometric features that will reappear later in the \(CR\) setting: a fibration by
rational curves, and an anti-holomorphic symmetry encoding the underlying real
geometry
\cite{AtiyahHitchinSinger1978,BurstallRawnsley1990,BurstallGuttRawnsley1993}.
\par
In the next subsections we shall pass from this even-dimensional
Riemannian picture to the odd-dimensional \(CR\) twistor geometry.

\subsection{$CR$ manifolds, Levi form, and flat models}

We now recall the basic notions from $CR$ geometry that will be used throughout
the paper. Since our main examples are real hypersurfaces in complex threefolds,
we shall freely pass between the abstract and the embedded points of view
\cite{DragomirTomassini2006,ChenShaw2001}.

A \emph{$CR$ structure} on a smooth manifold \(M\) of real dimension \(2n+k\)
is given by a rank-\(n\) complex subbundle
\[
H^{0,1}M\subset \C TM
\]
such that
\[
H^{0,1}M\cap \overline{H^{0,1}M}=\{0\},
\qquad
[\Gamma(H^{0,1}M),\Gamma(H^{0,1}M)]\subset \Gamma(H^{0,1}M).
\]
The pair \((M,H^{0,1}M)\) is then called a \(CR\) manifold of type \((n,k)\),
where \(n\) is the \(CR\) dimension and \(k\) the \(CR\) codimension.
Equivalently, one may describe the same structure by means of a rank-\(2n\)
real subbundle
\[
H_{\R}M=\Re\bigl(H^{1,0}M\oplus H^{0,1}M\bigr)\subset TM
\]
together with a complex structure
\[
J_M:H_{\R}M\to H_{\R}M,
\qquad
J_M^2=-\Id,
\]
whose \((-i)\)-eigenspace in \(\C\otimes H_{\R}M\) is \(H^{0,1}M\). In this
language, the integrability condition amounts to the vanishing of the Nijenhuis
tensor of \(J_M\) along the distribution \(H_{\R}M\).

When \(k=1\)  then \(M\) is said to be of
\emph{hypersurface type}; this is the one most relevant case, for historical reason.  If \(M\) is realized as a real hypersurface in a complex
manifold \((\widehat M,J)\), its induced \(CR\) structure is given by
\[
H_{\R}M=TM\cap J(TM),
\qquad
H^{0,1}M=\C TM\cap T^{0,1}\widehat M.
\]
Thus here the \(CR\) structure records the maximal complex subspaces of the tangent
spaces of \(M\). In the abstract setting, not every \(CR\) manifold is a priori
realizable in this way, but for the geometric problems considered here, since the construction is algebraic, it is in particular analytic, and  the embedded
picture will  serve as the 
model
\cite{AndreottiFredricks1979}.

A fundamental first-order invariant of a \(CR\) manifold is its \emph{Levi form} (see \cite{MariniMedoriNacinovich2021}). Let
\[
pr:\C TM\longrightarrow \C TM/(H^{1,0}M\oplus H^{0,1}M)
\]
be the quotient map. If \(L_1,L_2\in \Gamma(H^{1,0}M)\), the Levi form is defined by
\begin{equation}\label{eq:Levi-background}
\mathcal L(L_1,\overline{L}_2)
=
-\frac{i}{2}\,pr\bigl([L_1,\overline{L}_2]\bigr).
\end{equation}
This is a Hermitian form with values in the complex vector bundle
\[
\C TM/(H^{1,0}M\oplus H^{0,1}M).
\]
Its kernel is called the \emph{Levi kernel}; when this kernel is trivial, the \(CR\) structure
is called \emph{Levi nondegenerate}. In the hypersurface case, if \(\theta\) is a local
real \(1\)-form such that
\[
H_{\R}M=\ker\theta,
\]
then \eqref{eq:Levi-background} may be written in the familiar form
\[
\mathcal L(L_1,\overline{L}_2)
=
-\frac{i}{2}\,d\theta(L_1,\overline{L}_2).
\]
In particular, the signature of the Levi form is a local invariant of the
\(CR\) geometry.
For Levi--nondegenerate \(CR\) manifolds of hypersurface type, the local equivalence problem was solved by Cartan in dimension three and, in higher dimensions, by Tanaka's theory of algebraic prolongation and by the Chern--Moser normal form theory \cite{Tanaka1962,ChernMoser1974} ( see also
\cite{MariniMedoriNacinovich2020a} for algebraic prolongation).
A Levi--nondegenerate \(CR\) manifold of \(CR\) dimension \(n\), with Levi
signature \((p,q)\), \(p+q=n\), carries a canonical Cartan geometry of type
\[
(\PSU(p+1,q+1),P),
\]
where \(P\) is a real parabolic that stabilize  a complex null line. The corresponding
homogeneous space
\[
Q^{p,q}\cong \PSU(p+1,q+1)/P
\]
is the flat model of Levi--nondegenerate \(CR\) geometry of signature
\((p,q)\).

Concretely, it is realized as the projectivized null cone of a Hermitian
form of signature \((p+1,q+1)\), namely as the hyperquadric
\[
Q^{p,q}
=
\Bigl\{
[z_0:\cdots:z_{n+1}]\in \CP^{n+1}
\ \Big|\
|z_0|^2+\cdots+|z_p|^2
-|z_{p+1}|^2\cdots-|z_{n+1}|^2=0
\Bigr\}.\]
Its \(CR\) curvature vanishes identically, and conversely every Levi--nondegenerate
\(CR\) manifold of hypersurface type is locally equivalent to \(Q^{p,q}\) if and only if its canonical
Cartan curvature vanishes
\cite{Tanaka1962,ChernMoser1974}.

In the present article we are mainly interested in real dimension \(5\). In this case a non degenerate Levi form is either definite, of signature
\((2,0)\), or \emph{split}, of signature \((1,1)\). The corresponding flat models are the
two homogeneous hyperquadrics associated respectively with \(\SU(3,1)\) and
\(\SU(2,2)\), i.e. our main object \(\Qtwotwo\subset \CP^3\) of study.

\subsection{LeBrun's twistor $CR$ construction in dimension three}

We now turn our attention to the odd-dimensional twistor construction that is most relevant for
this paper, namely the one introduced by LeBrun for conformal Riemannian
\(3\)-manifolds \cite{LeBrun1984}.  This construction associates with every conformal
\(3\)-manifold \((\Sigma,[g])\) a natural \(5\)-dimensional \(CR\) manifold
\(M\), namely the projectivized bundle of nonzero complex null covectors.  It may be viewed as a
three-dimensional analogue of the classical twistor correspondence for
self-dual conformal \(4\)-manifolds.

Let \(g\) be a representative of the conformal class \([g]\). The metric \(g\)
induces a dual symmetric bilinear form \(g^{-1}\) on \(T^*\Sigma\), given in local
coordinates by the inverse matrix. We extend
\(g^{-1}\) by \(\C\)-bilinearity to a complex symmetric bilinear form on the
complexified cotangent bundle \(\C T^*\Sigma = T^*\Sigma\otimes_{\R}\C\).  
Consider the bundle of nonzero complex null covectors
\[
\widehat M
:=
\bigl\{
\zeta\in \C T^*\Sigma\setminus\{0\}
\;\big|\;
g^{-1}(\zeta,\zeta)=0
\bigr\}.
\]
Since the nullity condition is homogeneous, \(\widehat M\) is invariant under the
natural \(\C^*\)-action by fibrewise scalar multiplication, and we may form its
projectivization
\[
M
:=
\widehat M/\C^*
\subset
\PP(\C T^*\Sigma).
\]
Let
\[
\rho:\C T^*\Sigma\longrightarrow \Sigma
\]
denote the bundle projection. Since \(\widehat M\subset \C T^*\Sigma\setminus\{0\}\)
is \(\C^*\)-invariant, the restriction \(\rho|_{\widehat M}\) descends to a smooth map
\[
\pi_{CR}:M\longrightarrow \Sigma,
\qquad
\pi_{CR}([\zeta])=x
\quad\text{whenever } \zeta\in \C T_x^*\Sigma.
\]
Equivalently, \(\pi_{CR}\) is the restriction to \(M\subset \PP(\C T^*\Sigma)\) of the
natural projective-bundle projection
\[
\PP(\C T^*\Sigma)\longrightarrow \Sigma.
\]
We shall call \(M\) the \(CR\) twistor space associated with the conformal
\(3\)-manifold \((\Sigma,[g])\), and \(\pi_{CR}\) its twistor \(CR\) fibration.
For \(x\in\Sigma\), we write
\[
M_x:=\pi_{CR}^{-1}(x)
=
\PP\Bigl(
\{\zeta\in \C T_x^*\Sigma\setminus\{0\}\mid g_x^{-1}(\zeta,\zeta)=0\}
\Bigr).
\]
\begin{proposition}[{\cite[Section~1]{LeBrun1984}}]\label{prop:CR-twistor-basic-features}
Let \((\Sigma,[g])\) be a conformal Riemannian \(3\)-manifold, let
\[
M=\widehat M/\C^*
\subset \PP(\C T^*\Sigma)
\]
be its twistor \(CR\) space, and let
\[
\pi_{CR}:M\longrightarrow \Sigma
\]
be the associated twistor \(CR\) fibration. Then:
\begin{enumerate}
\item[\rm(i)] $M$ depends only on the conformal class $[g]$;
\item[\rm(ii)] for every \(x\in \Sigma\), the fibre
\[
M_x=\pi_{CR}^{-1}(x)
\]
is a smooth plane conic in \(\PP(\C T_x^*\Sigma)\cong \CP^2\), hence biholomorphic
to \(\CP^1\);
\item[\rm(iii)] in particular, $M$ is canonically foliated by Riemann spheres.
\end{enumerate}
\end{proposition}

The foliation by projective lines is one of the structural features of the
construction, and not merely an auxiliary by-product.  
\begin{remark}
In Section~\ref{TSRG} we recalled that the twistor fibre at a point
parametrizes orthogonal complex structures on the whole tangent space.  In
LeBrun's \(3\)-dimensional setting, the same idea survives on the two dimensional subspaces of \(T_x\Sigma\).

Indeed, let \(x\in \Sigma\) and let \([\zeta]\in M_x\).  Writing
\[
\zeta=\alpha+i\beta,
\qquad
\alpha,\beta\in T_x^*\Sigma,
\]
the nullity condition \(g^{-1}(\zeta,\zeta)=0\) is equivalent to
\[
g^{-1}(\alpha,\alpha)=g^{-1}(\beta,\beta),
\qquad
g^{-1}(\alpha,\beta)=0.
\]
Thus the duals \(\alpha^\sharp\) and \(\beta^\sharp\) span an oriented conformal
\(2\)-plane
\[
\xi_{[\zeta]}:=\Span(\alpha^\sharp,\beta^\sharp)\subset T_x\Sigma,
\]
where the orientation is the one determined by the ordered pair
\((\alpha^\sharp,\beta^\sharp)\).  The restriction of \(g\) to
\(\xi_{[\zeta]}\) defines a conformal class of Euclidean metrics on this plane,
and the chosen orientation determines, in the usual way, an almost complex
structure $J_{[\zeta]}\colon \xi_{[\zeta]}\to \xi_{[\zeta]}.$

In details,  if \(h_{[\zeta]}\) denotes the restriction of \(g\) to
\(\xi_{[\zeta]}\) and \(\omega_{[\zeta]}\) the positively oriented area form on
\(\xi_{[\zeta]}\), then \(J_{[\zeta]}\) is characterized by
\[
h_{[\zeta]}(J_{[\zeta]}u,v)=\omega_{[\zeta]}(u,v),
\qquad
u,v\in \xi_{[\zeta]}.
\]
In a positively oriented \(h_{[\zeta]}\)-orthonormal basis \((e_1,e_2)\) of
\(\xi_{[\zeta]}\), this simply means that
\[
J_{[\zeta]}e_1=e_2,
\qquad
J_{[\zeta]}e_2=-e_1.
\]

\end{remark}
The \(CR\) structure on \(M\) is induced from the tautological geometry of the
cotangent bundle.  Let \(\alpha\) denote the tautological \(1\)-form on
\(\C T^*\Sigma\), restricted to \(\widehat M\), and set
\[
\omega:=d\alpha.
\]
Then one considers the complex distribution
\[
H^{0,1}\widehat M:=\Ker \omega\cap \C T\widehat M.
\]
This distribution is \(\C^*\)-invariant, so it descends to a complex
rank-\(2\) distribution \(H^{0,1} M\) on the quotient \(M\).  LeBrun proved that this descended
distribution defines a nondegenerate \(CR\) structure on \(M\), of hypersurface type
and  signature of the Levi form \((1,1)\).  Moreover, the \(CR\) structure depends only on the conformal class
\([g]\), not on the chosen representative.

\begin{proposition}[{\cite[Theorem~1.2]{LeBrun1984}}]\label{prop:twistor-Levi-signature-background}
Let $(M,H^{0,1}M)$ be the twistor $CR$ manifold associated with a conformal
Riemannian $3$-manifold $(\Sigma,[g])$. Then $M$ is a Levi--nondegenerate
$CR$ manifold of hypersurface type, and its Levi form has signature
\((1,1)\) at every point.
\end{proposition}

There is also a useful alternative description when \(\Sigma\) is oriented.  After fixing a
metric \(g\in[g]\), the volume form identifies \(T\Sigma\) with \(\Lambda^2T^*\Sigma\), and the
space of null complex covectors can be identified with the unit sphere bundle
\[
S\Sigma=\{v\in T\Sigma\mid g(v,v)=1\}.
\]
Under this correspondence, the fibres of the twistor projection remain \(2\)-spheres,
and the orthogonal complement of the \(CR\) distribution is generated by the standard
contact form of the sphere bundle.

The construction captures the conformal geometry of the base in a precise rigid way: LeBrun showed
that realizability of the twistor $CR$ manifold (i.e.\ the existence of a local $CR$
embedding into a complex manifold) is closely tied to the real-analyticity of the
conformal structure \cite{LeBrun1984}.

Indeed, if \((\Sigma,[g])\) is locally conformally flat, then its twistor \(CR\) manifold is
locally \(CR\)-equivalent to the split hyperquadric
\[
\Qtwotwo
=
\bigl\{
[z_0:z_1:z_2:z_3]\in \CP^3
\;\big|\;
|z_0|^2+|z_1|^2-|z_2|^2-|z_3|^2=0
\bigr\},
\]
endowed with the \(CR\) structure induced from \(\CP^3\).  Conversely, the flatness
of the twistor \(CR\) structure reflects the conformal flatness of the base
\cite{Marugame2025}.  In this sense, \(\Qtwotwo\) is not merely a convenient
explicit example but it is the intrinsic flat $CR $ model in three-dimensional twistor geometry.

\begin{proposition}\label{prop:local-flatness-twistor}
Let $(\Sigma,[g])$ be a conformal Riemannian $3$-manifold, and let
$(M,H^{0,1}M)$ be its twistor $CR$ space. Then $M$ is locally
$CR$-equivalent to $\Qtwotwo$ if and only if $(\Sigma,[g])$ is locally
conformally flat.
\end{proposition}

\begin{proof}
By \cite[Theorem~4.8]{Marugame2025}, the Fefferman conformal structure of
$M$ is conformally flat if and only if $(\Sigma,[g])$ is conformally flat. By
Tanaka--Chern--Moser theory \cite{Tanaka1962,ChernMoser1974}, local $CR$
flatness for a Levi--nondegenerate $CR$ manifold with Levi form of signature $(1,1)$ is
equivalent to local $CR$ equivalence with $\Qtwotwo$. The conclusion follows
from Proposition~\ref{prop:twistor-Levi-signature-background}.
\end{proof}

\begin{corollary}\label{cor:global-flatness-twistor}
Let $(\Sigma,[g])$ be a compact simply connected locally conformally flat conformal
Riemannian $3$-manifold. Then its twistor $CR$ space is globally
$CR$-diffeomorphic to the hyperquadric $\Qtwotwo$.
\end{corollary}

\begin{proof}
By Proposition~\ref{prop:local-flatness-twistor}, the twistor $CR$ space
$M$ is locally $CR$-equivalent to $\Qtwotwo$. By
Lemma~\ref{lem:twistor-simply-connected}, $M$ is simply connected.

Since $(\Sigma,[g])$ is compact, simply connected, and locally conformally flat,
Kuiper's theorem implies that $(\Sigma,[g])$ is conformally diffeomorphic to the
round $3$-sphere \cite{Kuiper1949}. Therefore the associated twistor $CR$
space is the twistor space of the round conformal $3$-sphere, which is
globally identified with the split hyperquadric $\Qtwotwo$.
\end{proof}
We point out that in \cite[Theorem~4.8 and Section~5]{Marugame2025} Fefferman
metric of a twistor $CR$ manifold  is explicitly computed   and it is shown that its conformal geometry is governed by the geometry of the base.
\begin{proposition}\label{lem:twistor-simply-connected}
Let $(\Sigma,[g])$ be a conformal Riemannian $3$-manifold, and let
\[
\pi_{CR}:M\longrightarrow \Sigma
\]
be its twistor $CR$ fibration. If $\Sigma$ is simply connected, then $M$ is
simply connected as well.
\end{proposition}

\begin{proof}
The fibres of \(\pi_{CR}\) are \(\CP^1\) by
Proposition~\ref{prop:CR-twistor-basic-features}, hence simply connected.
The conclusion follows immediately from the long exact sequence of homotopy groups
of the fibration
\[
\CP^1\hookrightarrow M\overset{\pi_{CR}}{\longrightarrow} \Sigma,
\]
using \(\pi_1(\CP^1)=\pi_1(\Sigma)=0\).
\end{proof}

Then, even without the compactness assumption in Corollary~\ref{cor:global-flatness-twistor}, the locally conformally flat case implies that the twistor \(CR\) space is locally \(CR\)-equivalent to the flat model \(\Qtwotwo\) and if, in addition, it is simply connected, this local equivalence globalizes to a map into \(\Qtwotwo\), though not necessarily to a global \(CR\)-diffeomorphism onto \(\Qtwotwo\).

\subsection{The two ambient spaces: \texorpdfstring{$\CP^3$}{CP3} and the flag threefold}
\label{sec:ambient-models}

We now describe the two complex threefolds serving as ambient spaces throughout the
paper,  $\CP^3$ and the full flag threefold $\mathbb{F}$ . These are the basic algebraic
twistor threefolds associated with the rank-one geometries of $S^4$ and $\CP^2$,
providing the most concrete settings in which twistor fibres, real structures, and
embedded models can be studied explicitly.
Both spaces display the main geometric ingredients that recur throughout the paper, i.e.
fibrations by rational curves, anti-holomorphic symmetries, and homogeneous group
actions. The
computational core takes place in $\CP^3$, where $\Qtwotwo$ admits a particularly
simple projective realization; $\mathbb{F}$ serves as an essential complementary
model, helping to distinguish intrinsic features from those depending on the embedding.

\medskip
\noindent\textbf{The projective space model.}
Identifying $S^4$ with the quaternionic projective line $\HP^1$, one obtains the
quaternionic fibration \eqref{eq:Hopf-background}, whose fibres are projective
lines $\CP^1\subset\CP^3$, making $\CP^3$ a $\CP^1$-bundle over $S^4$
\cite{AtiyahHitchinSinger1978,Altavilla2025}. The quaternionic structure
on $\C^4$ is encoded by the anti-linear map $J(z)=S\overline{z}$, where
\[
S=
\begin{pmatrix}
0&-1&0&0\\
1&0&0&0\\
0&0&0&-1\\
0&0&1&0
\end{pmatrix},
\qquad J^2=-\Id,
\]
whose projectivization is the fixed-point-free anti-holomorphic involution
\begin{equation}\label{eq:j-section2}
j[z_0:z_1:z_2:z_3]=[-\overline z_1:\overline z_0:-\overline z_3:\overline z_2]
\end{equation}
\cite{SalamonViaclovsky2009,Altavilla2025}.
The geometric role of $j$ is to single out exactly the twistor fibres, i.e. a projective
line $\ell\subset\CP^3$ is a fibre of $\pi$ if and only if it is $j$-invariant,
equivalently of the form $\ell_z:=\PP(\Span_\C\{z,Jz\})$ for some
$z\in\C^4\setminus\{0\}$.

\medskip
\noindent\textbf{The flag model.}
The full flag threefold
\[
\mathbb{F}=\bigl\{(p,\ell)\in\CP^2\times(\CP^2)^*\mid p\in\ell\bigr\}
\cong \SL(3,\C)/B,
\]
where $B$ is a Borel subgroup, is the incidence variety of points and lines in $\CP^2$. It is the twistor
space of $\CP^2$ with its Fubini--Study conformal class, with twistor projection defined in Formula~\eqref{eq:flag-twistor-projection}
(see \cite{Altavilla2022}), whose fibres are again projective lines, making
$\mathbb{F}$ a $\CP^1$-bundle. The flag threefold is the counterpart of the
previous model and will serve as the natural home for the lift of the spherical $CR$
$3$-sphere in $\CP^2$, showing that the twistor phenomena we study are not exclusive
to the quaternionic setting.

\section{Embedded \texorpdfstring{$3$}{3}-spheres and their twistor lifts}\label{sec:embedded-spheres}
The \(3\)-sphere appears naturally in two related flag-geometric settings, on the
one hand, as a standard equatorial \(3\)-sphere inside \(S^4\)
equipped with the round metric, lifted through the
quaternionic twistor fibration \(\CP^3\to\mathbb{HP}^1\); on the other hand, as the flat
\(CR\) model in complex dimension two, realized as the standard \(CR\)
hypersurface in \(\CP^2\) and lifted to the corresponding flag threefold.
Although these two
constructions arise from different ambient geometries, the resulting twistor spaces are
closely related, in fact both encode the same flat Levi--nondegenerate $CR$ geometry of Levi form of 
signature $(1,1)$, and the comparison between them shows that this geometry is not
tied to a single ambient realization.

Computationally, one model is much more convenient, in fact the lift of the round $3$-sphere
to $\CP^3$ is precisely the   hyperquadric $\Qtwotwo$, and this is the main reason why the
subsequent analysis is carried out in $\CP^3$. We describe each embedded model and
its lift in turn, then compare the two and explain why they provide equivalent models
for the flat twistor $CR$ geometry studied in the rest of the paper.

\subsection{The round \texorpdfstring{$S^3\subset S^4$}{S3 in S4}}

The round $3$-sphere sits inside $S^4\simeq\HP^1$ as, e.g., the equatorial hypersphere
\begin{equation}\label{eq:S3-in-S4}
S^3=\bigl\{[q_0:q_1]\in\HP^1\;\big|\;\norm{q_0}=\norm{q_1}\bigr\},
\end{equation}
inheriting the round conformal structure from $S^4$. Writing $q_0=z_0+jz_1$ and
$q_1=z_2+jz_3$, so that $\norm{q_k}^2$ becomes the sum of squares of the
corresponding complex components, the inverse image of \eqref{eq:S3-in-S4} under
the twistor fibration \eqref{eq:Hopf-background} is
\begin{equation}\label{eq:Q22-as-lift}
\pi^{-1}(S^3)=\Bigl\{[z_0:z_1:z_2:z_3]\in\CP^3\;\Big|\;
|z_0|^2+|z_1|^2=|z_2|^2+|z_3|^2\Bigr\},
\end{equation}
which is precisely $\Qtwotwo$. Thus the flat twistor $CR$ model arises concretely as
the twistor lift of the round $3$-sphere: $\Qtwotwo$ is the projectivized bundle of
complex null covectors of $(S^3,[g_{\mathrm{rnd}}])$, and the restricted projection
$\pi|_{\Qtwotwo}:\Qtwotwo\to S^3$ is again a $\CP^1$-bundle.

The model \eqref{eq:Q22-as-lift} already exhibits the three features governing the
rest of the paper: $\Qtwotwo$ is embedded in $\CP^3$ and amenable to projective
methods; it carries a distinguished fibration by twistor fibres; and it inherits the
anti-holomorphic involution $j$ of \eqref{eq:j-section2}, which preserves $\Qtwotwo$
and encodes the underlying real twistor geometry.

The round \(3\)-sphere is not the only embedded \(3\)-manifold in \(S^4\) whose twistor lift is locally modeled on \(\Qtwotwo\); the following example shows that the global topology of the base may change without altering the local flat \(CR\) geometry upstairs.

\begin{example}
Besides the equatorial \(3\)-sphere, the \(4\)-sphere contains other natural embedded
\(3\)-manifolds whose twistor lifts are locally modeled on \(\Qtwotwo\).  A particularly
simple example is the \emph{generalized Clifford hypersurface} of type \((2,1)\),
obtained from the orthogonal splitting
\[
\R^5=\R^3\oplus \R^2.
\]
Writing a point of \(S^4\subset \R^5\) as
\[
x=(x_0,x_1,x_2,x_3,x_4),
\qquad
x_0^2+x_1^2+x_2^2+x_3^2+x_4^2=1,
\]
we set
\[T_{2,1}
:=
\Bigl\{
x\in S^4
\;\Big|\;
x_0^2+x_1^2+x_2^2=\frac12,
\quad
x_3^2+x_4^2=\frac12
\Bigr\}.\]
Equivalently,
\[
T_{2,1}
\cong 
S^2\!\left(\frac1{\sqrt2}\right)\times S^1\!\left(\frac1{\sqrt2}\right)
\subset \R^3\oplus \R^2,
\]
so \(T_{2,1}\) is diffeomorphic to \(S^2\times S^1\).  By analogy with the classical
terminology, one may also think of \(T_{2,1}\) as a Clifford hypersurface in \(S^4\).

We now describe \(M_T:=\pi^{-1}(T_{2,1})\subset \CP^3\) explicitly in homogeneous coordinates.  Let
\(q_0=z_0+jz_1\),
and \(q_1=z_2+jz_3\),
and set
\[
A:=|z_0|^2+|z_1|^2,
\qquad
B:=|z_2|^2+|z_3|^2,
\]
\[
\alpha:=z_2\overline z_0+z_3\overline z_1,
\qquad
\beta:=z_0z_3-z_1z_2.
\]
Then the twistor map may be written as
\[\pi([z_0:z_1:z_2:z_3])
=
\frac{1}{A+B}
\bigl(
A-B,\,
2\Re\alpha,\,
2\Im\alpha,\,
2\Re\beta,\,
2\Im\beta
\bigr),\]
which indeed takes values in \(S^4\subset \R^5\), since
\[
(A-B)^2+4|\alpha|^2+4|\beta|^2=(A+B)^2.
\]
Therefore the condition \(\pi([z])\in T_{2,1}\) is equivalent to
\[
\left(\frac{A-B}{A+B}\right)^2
+
\left(\frac{2\Re\alpha}{A+B}\right)^2
+
\left(\frac{2\Im\alpha}{A+B}\right)^2
=
\frac12,
\]
or, equivalently,
\[M_T
=
\left\{
[z_0:z_1:z_2:z_3]\in \CP^3
\ \middle|\
2\bigl((A-B)^2+4|\alpha|^2\bigr)=(A+B)^2
\right\}.\]
Using the identity above, one may also rewrite this as
\[M_T
=
\left\{
[z_0:z_1:z_2:z_3]\in \CP^3
\ \middle|\
(A-B)^2+4|\alpha|^2=4|\beta|^2
\right\}.\]
Thus \(M_T\) is a real hypersurface in \(\CP^3\), given by a single explicit
homogeneous equation.

On the affine chart \(U_0=\{z_0\neq 0\}\), with coordinates
\(u_1=\frac{z_1}{z_0},
u_2=\frac{z_2}{z_0},
u_3=\frac{z_3}{z_0}\),
one has
\[
A=1+|u_1|^2,
\qquad
B=|u_2|^2+|u_3|^2,
\qquad
\alpha=u_2+u_3\overline u_1,
\qquad
\beta=u_3-u_1u_2,
\]
and therefore \(M_T\cap U_0\) is given by
\[2\Bigl(
\bigl(1+|u_1|^2-|u_2|^2-|u_3|^2\bigr)^2
+
4|u_2+u_3\overline u_1|^2
\Bigr)
=
\bigl(1+|u_1|^2+|u_2|^2+|u_3|^2\bigr)^2.\]

The twistor \(CR\) manifold \(M_T\) is locally \(CR\)-equivalent to \(\Qtwotwo\), but
not globally.  Indeed, the product metric on \(S^2\times \R\) is locally conformally
flat, since if \(r=e^t\) then the Euclidean metric on \(\R^3\setminus\{0\}\) may be
written as
\[
dr^2+r^2g_{S^2}
=
e^{2t}\bigl(dt^2+g_{S^2}\bigr).
\]
Hence \(dt^2+g_{S^2}\) is conformal to the flat Euclidean metric on
\(\R^3\setminus\{0\}\).  Since \(S^2\times S^1\) is locally isometric to \(S^2\times \R\),
it follows that \(T_{2,1}\) is locally conformally flat.  By LeBrun's theorem, its
twistor \(CR\) space is therefore locally \(CR\)-equivalent to the flat model
\(\Qtwotwo\).

On the other hand, \(M_T\) cannot be globally \(CR\)-equivalent to \(\Qtwotwo\).
Indeed, the fibres of the restricted projection
\(\pi|_{M_T}:M_T\longrightarrow T_{2,1}\)
are projective lines \(\CP^1\), hence simply connected, and the long exact sequence of
homotopy groups gives
\[
\pi_1(M_T)\cong \pi_1(T_{2,1})\cong \Z.
\]
By contrast, \(\Qtwotwo\) is the twistor lift of \(S^3\), so
\[
\pi_1(\Qtwotwo)=0.
\]
Thus \(M_T\) and \(\Qtwotwo\) are locally \(CR\)-equivalent but globally distinct.
\end{example}

\subsection{The spherical \texorpdfstring{$CR$}{CR} \texorpdfstring{$3$}{3}-sphere in \texorpdfstring{$\CP^2$}{CP2}}\label{3sph}

We now consider the second natural embedded model of the \(3\)-sphere, namely the
standard spherical \(CR\) hypersurface in \(\CP^2\).  Let
\[
h_{1,2}(z,z)=|z_0|^2-|z_1|^2-|z_2|^2
\]
be the Hermitian form of signature \((1,2)\) on \(\C^3\).  The corresponding projective
hyperquadric
\begin{equation}\label{eq:spherical-CR-sphere}
S^3_{\mathrm{sph}}
=
\bigl\{
[z_0:z_1:z_2]\in \CP^2
\;\big|\;
|z_0|^2=|z_1|^2+|z_2|^2
\bigr\}
\end{equation}
is a smooth real hypersurface in \(\CP^2\), endowed with the \(CR\) structure induced
from the ambient complex structure.  In the affine chart \(z_0=1\), it is given by
\[
|z_1|^2+|z_2|^2=1,
\]
so it is diffeomorphic to the standard \(3\)-sphere.  Its \(CR\) structure is the usual
spherical one, and its full \(CR\) automorphism group is \(\PSU(1,2)\) \cite{Cartan1925}.

From the homogeneous point of view, \eqref{eq:spherical-CR-sphere} is the unique
closed \(\SU(1,2)\)-orbit in \(\CP^2\).  Thus, just as the round \(3\)-sphere appears as
a distinguished hypersphere in \(S^4\), the spherical \(CR\) \(3\)-sphere appears here as
a distinguished homogeneous hypersurface in the complex projective plane.

Its twistor lift lives in the flag threefold $\mathbb{F}$, via the twistor projection
$\pi_{\mathbb{F}}$ introduced in \eqref{eq:flag-twistor-projection}.
We may therefore define the lift of the spherical \(CR\) sphere by
\[
M_{\mathbb F}
:=
\pi_{\mathbb F}^{-1}(S^3_{\mathrm{sph}})
\subset \mathbb F.
\]
Equivalently, \(M_{\mathbb F}\) admits the following explicit description inside the
incidence variety \(\mathbb F\).  Writing
\(p=[p_0:p_1:p_2]\in \CP^2,
\ell=[\ell_0:\ell_1:\ell_2]\in (\CP^2)^*\),
with the incidence relation
\(p_0\ell_0+p_1\ell_1+p_2\ell_2=0\),
and using the explicit formula for
\(\pi_{\mathbb F}(p,\ell)
\) in  \eqref{eq:flag-twistor-projection},
the condition \(\pi_{\mathbb F}(p,\ell)\in S^3_{\mathrm{sph}}\) is equivalent to
\begin{equation}\label{eq:MF-explicit}
\bigl|\overline p_1\ell_2-\overline p_2\ell_1\bigr|^2
=
\bigl|\overline p_2\ell_0-\overline p_0\ell_2\bigr|^2
+
\bigl|\overline p_0\ell_1-\overline p_1\ell_0\bigr|^2.
\end{equation}
Hence
\[M_{\mathbb F}
=
\left\{
(p,\ell)\in \mathbb F
\ \middle|\
\bigl|\overline p_1\ell_2-\overline p_2\ell_1\bigr|^2
=
\bigl|\overline p_2\ell_0-\overline p_0\ell_2\bigr|^2
+
\bigl|\overline p_0\ell_1-\overline p_1\ell_0\bigr|^2
\right\}.\]
Thus \(M_{\mathbb F}\) is a real hypersurface in the flag threefold, explicitly defined
by the incidence relation together with the single real equation
\eqref{eq:MF-explicit}.
As in the quaternionic case, the fibres of the restricted projection
\[
\pi_{\mathbb F}|_{M_{\mathbb F}}:M_{\mathbb F}\longrightarrow S^3_{\mathrm{sph}}
\]
are projective lines, so \(M_{\mathbb F}\) is again a \(\CP^1\)-bundle over a
\(3\)-sphere.

There is also a distinguished homogeneous section of this bundle, obtained by sending
an isotropic line \(p\in S^3_{\mathrm{sph}}\subset \CP^2\) to its Hermitian orthogonal
complement \(p^\perp\).  More explicitly,
\[
\Sigma_{\mathbb F}
=
\bigl\{
(p,p^\perp)\in \mathbb F
\mid
p\in S^3_{\mathrm{sph}}
\bigr\}
\subset M_{\mathbb F}.
\]
This submanifold is naturally identified with the spherical \(CR\) \(3\)-sphere itself,
and may be viewed as the flag-theoretic analogue of the distinguished real locus in the
quaternionic model.

Thus $M_{\mathbb{F}}$ provides a second explicit twistor lift of a $3$-sphere,
complementary to $\Qtwotwo$. Although they arise from different ambient geometries,
both encode the same flat Levi--nondegenerate $CR$ geometry of Levi form of signature $(1,1)$,
they are two
realizations of the same underlying model, seen from the conformal-quaternionic and
the projective-$CR$ side respectively, confirming that this geometry is intrinsic to
the lifted $CR$ structure and not an artifact of the embedding.

\begin{remark}\label{rem:orbits}
The two embedded models of $S^3$ described above both arise as the unique closed
real orbit in a flag manifold (see  \cite{MariniMedoriNacinovich2023} and the references therein).

On the conformal side, let $SO_0(4,1)$ act on $S^4$ preserving the quadratic form
$q(x)=x_1^2+x_2^2+x_3^2+x_4^2-x_5^2$. Then
\[
S^4=\mathcal O_+\sqcup \Sigma\sqcup \mathcal O_-,
\qquad
\mathcal O_\pm=\{x\in S^4\mid q(x)\gtrless 0\},
\quad
\Sigma=\{x\in S^4\mid q(x)=0\}.
\]
The condition $q(x)=0$ together with $\sum_{j=1}^5 x_j^2=1$ forces
$x_1^2+\cdots+x_4^2=\tfrac12$, so $\Sigma\cong S^3$ is the unique closed orbit,
homogeneous and locally conformally flat \cite{KrotzSchlichtkrull2014}.

On the complex projective side, let \(SU(1,2)\) act on \(\CP^2\) preserving
\[
h_{1,2}(z,z)=|z_0|^2-|z_1|^2-|z_2|^2.
\]
Then
\[
\CP^2=\Omega_+\sqcup Q^{1,2}\sqcup\Omega_-,
\]
where
\[
\Omega_+:=\{[z]\in\CP^2\mid h_{1,2}(z,z)>0\},\qquad\Omega_-:=\{[z]\in\CP^2\mid h_{1,2}(z,z)<0\},
\]
and
\[
Q^{1,2}:=\{[z]\in\CP^2\mid h_{1,2}(z,z)=0\},.
\]
Thus \(\Omega_+\) and \(\Omega_-\) are the two open \(SU(1,2)\)-orbits, while
\(Q^{1,2}\) is the unique closed orbit. In the chart \(z_0=1\), the equation of
\(Q^{1,2}\) becomes
\[
|z_1|^2+|z_2|^2=1,
\]
so \(Q^{1,2}\cong S^3\), endowed with its standard flat \(CR\) structure.
The twistor pull-back of $\Sigma$ to $\CP^3$ is $\Qtwotwo$, which is thus the unique
closed $SU(2,2)$-orbit in $\CP^3$; the pull-back of $Q^{1,2}$ to the flag threefold
is $M_{\mathbb F}$, the unique closed $SU(1,2)$-orbit there. Both are locally
$CR$-equivalent to $\Qtwotwo$.
\end{remark}

\section{Explicit \texorpdfstring{$CR$}{CR} geometry of \texorpdfstring{$\Qtwotwo\subset \CP^3$}{Q22 in CP3}}\label{sec:explicit-CR} 
We now develop the explicit \(CR\) geometry 
of \(\Qtwotwo\subset\CP^3\), defining 
equations in affine charts, local 
generators of the \(CR\) bundle, the 
contact form, and a direct verification 
that the Levi form has constant split 
signature \((1,1)\). Beyond providing the 
technical backbone for the later sections, 
these computations also serve a conceptual 
purpose, namely to show explicitly how the 
quaternionic involution \(j\), the twistor 
fibration, and the induced \(CR\) 
distribution fit together inside the 
projective model, thereby making precise 
the interaction between the intrinsic flat 
\(CR\) geometry and the extrinsic 
projective geometry of \(\CP^3\). 
\subsection{Defining equations and affine charts} 
We now describe the flat model \(\Qtwotwo\) 
explicitly as a real hypersurface in 
projective space. Let \[ H=\diag(1,1,-1,-1) \] and consider the Hermitian form \[ h(z,w)=z^*Hw, \qquad z,w\in \C^4. \] Then 
the hyperquadric is the projectivized null 
cone of \(h\), namely 
\begin{equation}\label{eq:Q22-homogeneous} \Qtwotwo = \bigl\{ [z_0:z_1:z_2:z_3]\in \CP^3 \;\big|\; h(z,z)=0 \bigr\}. 
\end{equation} 
Since the defining equation is homogeneous 
of bidegree \((1,1)\), the hypersurface 
\(\Qtwotwo\) is well defined in projective 
space and is invariant under the natural action of the projective unitary group of the form \(h\). 
The \(CR\) structure on \(\Qtwotwo\) is the one induced from the ambient complex structure of \(\CP^3\). Thus, if \(J_{\CP^3}\) denotes the complex structure of \(\CP^3\), the underlying real \(CR\) distribution is \[ H_{\R}\Qtwotwo = T\Qtwotwo\cap J_{\CP^3}(T\Qtwotwo), \] and the bundle \(T^{0,1}\Qtwotwo\) is the \((-i)\)-eigenspace of the induced complex structure on \(H_{\R}\Qtwotwo\). In order to compute this structure explicitly, we pass to affine charts. For \(j=0,1,2,3\), let \[ U_j=\{z_j\neq 0\}\subset \CP^3 \] be the standard affine chart. Since \(\Qtwotwo\) is invariant under permutations of the positive coordinates \((z_0,z_1)\) and separately of the negative coordinates \((z_2,z_3)\), it will be enough for most purposes to work in one chart from each sign block. The most convenient choice will be the chart \[ U_0=\{z_0\neq 0\}, \] with affine coordinates \[ u_1=\frac{z_1}{z_0}, \qquad u_2=\frac{z_2}{z_0}, \qquad u_3=\frac{z_3}{z_0}. \] In these coordinates, \eqref{eq:Q22-homogeneous} becomes \begin{equation}\label{eq:Q22-affine-U0} \Qtwotwo\cap U_0 = \bigl\{ (u_1,u_2,u_3)\in \C^3 \;\big|\; 1+|u_1|^2-|u_2|^2-|u_3|^2=0 \bigr\}, \end{equation} so that on \(U_0\) the hypersurface is defined by the real-valued function \[ \rho_0(u,\overline u) = 1+|u_1|^2-|u_2|^2-|u_3|^2. \]
For comparison we also record the chart \[ U_3=\{z_3\neq 0\}, \] with affine coordinates \[ w_0=\frac{z_0}{z_3}, \qquad w_1=\frac{z_1}{z_3}, \qquad w_2=\frac{z_2}{z_3}. \] There \eqref{eq:Q22-homogeneous} becomes \begin{equation}\label{eq:Q22-affine-U3} \Qtwotwo\cap U_3 = \bigl\{ (w_0,w_1,w_2)\in \C^3 \;\big|\; |w_0|^2+|w_1|^2-|w_2|^2-1=0 \bigr\}, \end{equation} that is, \(\Qtwotwo\cap U_3\) is defined by \[ \rho_3(w,\overline w) = |w_0|^2+|w_1|^2-|w_2|^2-1. \] 
The affine forms \eqref{eq:Q22-affine-U0} and \eqref{eq:Q22-affine-U3} make the split nature of the hypersurface completely transparent, in each chart \(\Qtwotwo\) appears as an indefinite hyperquadric in \(\C^3\), and this is the local manifestation of the fact that \(\Qtwotwo\) is the flat Levi--nondegenerate \(CR\) model with Levi form of signature \((1,1)\). For later use we note that the defining equation \eqref{eq:Q22-homogeneous} is manifestly invariant under the involution \(j\) of \eqref{eq:j-section2}, which therefore preserves \(\Qtwotwo\). The affine models above thus describe the local form of a distinguished \(CR\) hypersurface in projective twistor space. In the next subsection we compute explicit local generators of the bundle \(T^{0,1}\Qtwotwo\). 
\subsection{Local generators of the \texorpdfstring{$CR$}{CR} bundle} We now describe the bundle \(T^{0,1}\Qtwotwo\) explicitly in the affine charts introduced above. \(\Qtwotwo\) is a real hypersurface in the complex threefold \(\CP^3\) and  its \(CR\) dimension is \(2\). Accordingly, in each affine chart the bundle \(T^{0,1}\Qtwotwo\) is a rank-\(2\) complex subbundle of the ambient antiholomorphic tangent bundle. We begin with the chart \(U_0\), where \[ (u_1,u_2,u_3)=\left(\frac{z_1}{z_0},\frac{z_2}{z_0},\frac{z_3}{z_0}\right), \qquad \rho_0(u,\overline u)=1+|u_1|^2-|u_2|^2-|u_3|^2. \] An ambient antiholomorphic vector field has the form \[ \overline L = c_1\frac{\partial}{\partial \overline u_1} + c_2\frac{\partial}{\partial \overline u_2} + c_3\frac{\partial}{\partial \overline u_3}, \] where \(c_1,c_2,c_3\) are smooth complex-valued functions. Such a vector field is tangent to \(\Qtwotwo\) if and only if \[ \overline L(\rho_0)=0, \] that is, \begin{equation}\label{eq:tangency-U0-CR} c_1u_1-c_2u_2-c_3u_3=0. \end{equation} Hence \[ T^{0,1}(\Qtwotwo\cap U_0) = \left\{ c_1\frac{\partial}{\partial \overline u_1} + c_2\frac{\partial}{\partial \overline u_2} + c_3\frac{\partial}{\partial \overline u_3} \;\middle|\; c_1u_1-c_2u_2-c_3u_3=0 \right\}. \] A convenient generating system is given by the three ambient antiholomorphic vector fields \[ \overline M_{12} := u_2\frac{\partial}{\partial \overline u_1} + u_1\frac{\partial}{\partial \overline u_2}, \qquad \overline M_{13} := u_3\frac{\partial}{\partial \overline u_1} + u_1\frac{\partial}{\partial \overline u_3}, \qquad \overline M_{23} := u_3\frac{\partial}{\partial \overline u_2} - u_2\frac{\partial}{\partial \overline u_3}. \] Indeed, a direct computation shows that \[ \overline M_{12}(\rho_0)= \overline M_{13}(\rho_0)= \overline M_{23}(\rho_0)=0, \] so these vector fields are tangent to \(\Qtwotwo\). They satisfy the relation 
\[ u_3\,\overline M_{12} - u_2\,\overline M_{13} - u_1\,\overline M_{23} = 0,\]  
and therefore generate a rank-\(2\) complex bundle, as expected. On smaller open sets one may choose an actual local frame consisting of two generators. For instance, on \[ U_0\cap\{u_3\neq 0\}, \] equation \eqref{eq:tangency-U0-CR} allows us to solve for the coefficient of \(\partial/\partial \overline u_3\), giving the local frame \[ \overline Z_1 := \frac{\partial}{\partial \overline u_1} + \frac{u_1}{u_3}\frac{\partial}{\partial \overline u_3}, \qquad \overline Z_2 := \frac{\partial}{\partial \overline u_2} - \frac{u_2}{u_3}\frac{\partial}{\partial \overline u_3}. \] Indeed, \[ \overline Z_1(\rho_0)=\overline Z_2(\rho_0)=0, \] and every local section of \(T^{0,1}\Qtwotwo\) on this open set can be written uniquely as a complex linear combination of \(\overline Z_1\) and \(\overline Z_2\). Analogous frames are obtained on the open sets where \(u_1\neq 0\) or \(u_2\neq 0\). For comparison, on the chart \(U_3\), where \[ (w_0,w_1,w_2)=\left(\frac{z_0}{z_3},\frac{z_1}{z_3},\frac{z_2}{z_3}\right), \qquad \rho_3(w,\overline w)=|w_0|^2+|w_1|^2-|w_2|^2-1, \] an ambient antiholomorphic vector field \[ \overline L = c_0\frac{\partial}{\partial \overline w_0} + c_1\frac{\partial}{\partial \overline w_1} + c_2\frac{\partial}{\partial \overline w_2} \] is tangent to \(\Qtwotwo\) if and only if 
\[
c_0w_0+c_1w_1-c_2w_2=0.\] 
A convenient generating system is then \[ \overline L_{01} := w_1\frac{\partial}{\partial \overline w_0} - w_0\frac{\partial}{\partial \overline w_1}, \qquad \overline L_{02} := w_2\frac{\partial}{\partial \overline w_0} + w_0\frac{\partial}{\partial \overline w_2}, \qquad \overline L_{12} := w_2\frac{\partial}{\partial \overline w_1} + w_1\frac{\partial}{\partial \overline w_2}, \] which satisfy \[ \overline L_{01}(\rho_3)= \overline L_{02}(\rho_3)= \overline L_{12}(\rho_3)=0 \] and the relation 
\[
 w_1\,\overline L_{02} - w_0\,\overline L_{12} - w_2\,\overline L_{01} = 0. 
\] 
On \(U_3\cap\{w_2\neq 0\}\) one may take the local frame \[ \overline Y_0 := \frac{\partial}{\partial \overline w_0} + \frac{w_0}{w_2}\frac{\partial}{\partial \overline w_2}, \qquad \overline Y_1 := \frac{\partial}{\partial \overline w_1} + \frac{w_1}{w_2}\frac{\partial}{\partial \overline w_2}. \] 
\subsection{Contact form and Levi form} We now compute the contact form and the Levi form of the hypersurface \(\Qtwotwo\subset \CP^3\) in the affine charts introduced above. Since the \(CR\) structure is induced from the ambient complex structure of projective space, a natural contact form is obtained in the standard way from a local defining function \cite{DragomirTomassini2006,ChenShaw2001,ChernMoser1974}. We begin on the chart \(U_0\), where \[ \Qtwotwo\cap U_0 = \bigl\{ (u_1,u_2,u_3)\in \C^3 \;\big|\; \rho_0(u,\overline u)=1+|u_1|^2-|u_2|^2-|u_3|^2=0 \bigr\}. \] A compatible contact form is \[\theta_0 = \frac{i}{2}\bigl(\partial \rho_0-\bar\partial \rho_0\bigr)\big|_{T\Qtwotwo}.
\]
Since 
\[ \partial \rho_0 = \overline u_1\,du_1-\overline u_2\,du_2-\overline u_3\,du_3, \qquad \bar\partial \rho_0 = u_1\,d\overline u_1-u_2\,d\overline u_2-u_3\,d\overline u_3, \] 
we obtain 
\[
\theta_0 = \frac{i}{2} \Bigl( \overline u_1\,du_1-\overline u_2\,du_2-\overline u_3\,du_3 - u_1\,d\overline u_1+u_2\,d\overline u_2+u_3\,d\overline u_3 \Bigr)\Big|_{T\Qtwotwo}. 
\]
By construction, \[ H_{\R}\Qtwotwo=\Ker \theta_0, \] and the bundle \(T^{0,1}\Qtwotwo\) is the \((-i)\)-eigenspace of the induced complex structure on \(H_{\R}\Qtwotwo\). 
Although we describe the \(CR\) structure primarily in terms of the bundle
\(T^{0,1}\Qtwotwo\), for the Levi form we adopt the standard convention and evaluate it
on the conjugate bundle $T^{1,0}\Qtwotwo=\overline{T^{0,1}\Qtwotwo}.$ Thus, if \(L,M\in \Gamma\bigl(T^{1,0}\Qtwotwo\bigr)\), the Levi form is given by
\[
\mathcal L(L,\overline M)
=
-\frac{i}{2}\,d\theta_0(L,\overline M).
\]
Since
\[
d\theta_0
=
i\bigl(
du_1\wedge d\overline u_1
-
du_2\wedge d\overline u_2
-
du_3\wedge d\overline u_3
\bigr)\Big|_{T\Qtwotwo},
\]
it follows that
\begin{equation}\label{eq:Levi-U0-general}
\mathcal L(L,\overline M)
=
\frac12
\Bigl(
du_1(L)\,d\overline u_1(\overline M)
-
du_2(L)\,d\overline u_2(\overline M)
-
du_3(L)\,d\overline u_3(\overline M)
\Bigr).
\end{equation}
In particular, the Levi form is simply the restriction to the complex tangent bundle of
the ambient Hermitian form with coefficients \((1,-1,-1)\).

Let us evaluate \(\mathcal L\) on the local frame of \(T^{1,0}\Qtwotwo\) conjugate to
\(\{\overline Z_1,\overline Z_2\}\), namely
\[
Z_1
=
\frac{\partial}{\partial u_1}
+
\frac{\overline u_1}{\overline u_3}\frac{\partial}{\partial u_3},
\qquad
Z_2
=
\frac{\partial}{\partial u_2}
-
\frac{\overline u_2}{\overline u_3}\frac{\partial}{\partial u_3},
\]
defined on \(U_0\cap\{u_3\neq 0\}\). Using \eqref{eq:Levi-U0-general}, we find
\[
\mathcal L(Z_1,\overline Z_1)
=
\frac12
\left(
1-\frac{|u_1|^2}{|u_3|^2}
\right),
\qquad
\mathcal L(Z_2,\overline Z_2)
=
-\frac12
\left(
1+\frac{|u_2|^2}{|u_3|^2}
\right),
\]
and
\[
\mathcal L(Z_1,\overline Z_2)
=
\frac12\,\frac{u_2\overline u_1}{|u_3|^2},
\qquad
\mathcal L(Z_2,\overline Z_1)
=
\frac12\,\frac{u_1\overline u_2}{|u_3|^2}.
\]
Therefore, in the frame \(\{Z_1,Z_2\}\), the Levi matrix is
\begin{equation}\label{eq:Levi-matrix-U0}
[\mathcal L]_{\{Z_1,Z_2\}}
=
\frac{1}{2|u_3|^2}
\begin{pmatrix}
|u_3|^2-|u_1|^2 & u_2\overline u_1\\[4pt]
u_1\overline u_2 & -(|u_3|^2+|u_2|^2)
\end{pmatrix}.
\end{equation}
Its determinant is
\begin{align*}
\det[\mathcal L]_{\{Z_1,Z_2\}}
&=
\frac{1}{4|u_3|^4}
\Bigl(
-(|u_3|^2-|u_1|^2)(|u_3|^2+|u_2|^2)-|u_1|^2|u_2|^2
\Bigr)\\
&=
-\frac{1}{4|u_3|^4}\,
|u_3|^2\bigl(|u_3|^2+|u_2|^2-|u_1|^2\bigr)\\
&=
-\frac{1}{4|u_3|^2}<0,
\end{align*}
where in the last step we used the defining equation $|u_3|^2=1+|u_1|^2-|u_2|^2.$

Hence the Levi form is nondegenerate and has split signature \((1,1)\) at every point
of this open set. Since the signature is locally constant on the connected hypersurface
\(\Qtwotwo\), it follows that \(\Qtwotwo\) is everywhere Levi nondegenerate of
signature \((1,1)\).

The same conclusion is obtained in the affine chart \(U_3\), where
\[
\rho_3(w,\overline w)=|w_0|^2+|w_1|^2-|w_2|^2-1.
\]
The associated contact form is
\[
\theta_3
=
\frac{i}{2}
\Bigl(
\overline w_0\,dw_0+\overline w_1\,dw_1-\overline w_2\,dw_2
-
w_0\,d\overline w_0-w_1\,d\overline w_1+w_2\,d\overline w_2
\Bigr)\Big|_{T\Qtwotwo},
\]
and
\[
d\theta_3
=
i\bigl(
dw_0\wedge d\overline w_0
+
dw_1\wedge d\overline w_1
-
dw_2\wedge d\overline w_2
\bigr)\Big|_{T\Qtwotwo}.
\]
Thus the Levi form is again the restriction of an ambient Hermitian form, this time
with coefficients \((1,1,-1)\), subject to the complex tangency condition
\[
c_0\overline w_0+c_1\overline w_1-c_2\overline w_2=0
\]
for \(T^{1,0}\Qtwotwo\), equivalently
\[
c_0w_0+c_1w_1-c_2w_2=0
\]
for \(T^{0,1}\Qtwotwo\). For example, in the local frame
\[
Y_0
=
\frac{\partial}{\partial w_0}
+
\frac{\overline w_0}{\overline w_2}\frac{\partial}{\partial w_2},
\qquad
Y_1
=
\frac{\partial}{\partial w_1}
+
\frac{\overline w_1}{\overline w_2}\frac{\partial}{\partial w_2},
\]
defined on \(U_3\cap\{w_2\neq 0\}\), the Levi matrix is
\[
[\mathcal L]_{\{Y_0,Y_1\}}
=
\frac{1}{2|w_2|^2}
\begin{pmatrix}
|w_2|^2-|w_0|^2 & -w_1\overline w_0\\[4pt]
-\overline w_1w_0 & |w_2|^2-|w_1|^2
\end{pmatrix},
\]
whose determinant is again \(-1/(4|w_2|^2)\). This confirms, in a second affine
realization, that the Levi signature is everywhere \((1,1)\).

We summarize the discussion in the following proposition.

\begin{proposition}
The hypersurface \(\Qtwotwo\subset \CP^3\), endowed with the \(CR\) structure induced
from the ambient complex structure, is a Levi--nondegenerate \(CR\) manifold of
hypersurface type and constant Levi signature \((1,1)\).
\end{proposition}

\begin{proof}
The local contact forms \(\theta_0\) and \(\theta_3\) above define the same
cooriented \(CR\) structure on overlapping charts. The explicit computation of the
Levi matrix in \eqref{eq:Levi-matrix-U0} shows that the determinant is everywhere
negative, hence the Levi form is nondegenerate and indefinite. Since the \(CR\)
dimension is \(2\), the only possible indefinite signature is \((1,1)\).
\end{proof}

On the affine chart \(U_0\), the hypersurface \(\Qtwotwo\) is given by
\[
1+|u_1|^2-|u_2|^2-|u_3|^2=0.
\]
This is precisely the indefinite hyperquadric considered by Sommer
\cite{Sommer1959} as a basic example of a Levi--nondegenerate \(CR\) manifold
which nevertheless admits a foliation by complex curves. As we shall see below,
this foliation is related to the twistor fibres.

\subsection{Twistor fibres and the involution \texorpdfstring{$j$}{j}}

After the local \(CR\) computations, we now return to the global twistor geometry of
the embedding \(\Qtwotwo\subset\CP^3\). The first issue is to identify, among the
projective lines contained in \(\Qtwotwo\), those that come from the ambient twistor
fibration \(\pi:\CP^3\to\HP^1\). The point of this subsection is that these
distinguished lines are characterized exactly by \(j\)-invariance.

As recalled in Section~\ref{sec:ambient-models}, the fibre of $\pi$ through
$[z]\in\CP^3$ is $\ell_z=\PP(\Span_\C\{z,Jz\})$; since $J$ preserves this plane,
every fibre is $j$-invariant, and conversely every $j$-invariant line arises this way
\cite{AtiyahHitchinSinger1978,SalamonViaclovsky2009}.
In particular, a projective line $\ell\subset\CP^3$ is a twistor fibre of
$\pi:\CP^3\to\HP^1$ if and only if it is $j$-invariant.

We now restrict to $\Qtwotwo$. Since $h(Jz,Jw)=\overline{h(z,w)}$ and $h(z,Jz)=0$,
every $h$-null vector $z$ satisfies $h(Jz,Jz)=0$, so the entire fibre $\ell_z$ lies
in $\Qtwotwo$ whenever $[z]\in\Qtwotwo$.

\begin{proposition}
The fibres of $\pi|_{\Qtwotwo}:\Qtwotwo\to S^3$ are exactly the $j$-invariant
projective lines contained in $\Qtwotwo$. Equivalently, for every $[z]\in\Qtwotwo$,
\[
\pi^{-1}(\pi([z]))\cap\Qtwotwo=\PP\bigl(\Span_{\C}\{z,Jz\}\bigr).
\]
\end{proposition}

The twistor fibres are thus singled out by the single condition of \(j\)-invariance.
This gives the first projective characterization of the twistor geometry of
\(\Qtwotwo\), and it will serve as the basic reference point for the classification
of all projective lines contained in \(\Qtwotwo\) in the next section.

\section{Projective and twistor symmetries of \texorpdfstring{$\Qtwotwo$}{Q22}}\label{sec:symmetries}

We now turn to the symmetry theory of $\Qtwotwo\subset\CP^3$. As the flat
Levi--nondegenerate model of signature $(1,1)$, its full $CR$ automorphism group is
$\PSU(2,2)$, the standard symmetry group of the   hyperquadric in the
Tanaka--Chern--Moser framework. However, since
$\Qtwotwo$ is the twistor lift of the round $3$-sphere, not every projective symmetry
is relevant for the twistor geometry: the involution $j$ singles out the subgroup of
\emph{twistor-compatible} automorphisms, those commuting with $j$ and hence
preserving the family of twistor fibres
\cite{SalamonViaclovsky2009,Altavilla2025}. We first describe the full
projective stabilizer of $\Qtwotwo$, then isolate this twistor-compatible subgroup.
The ambient projective realization and the explicit matrix form of $j$ reduce both
problems to concrete linear algebra, and the resulting groups descend to conformal
symmetries of the base $3$-sphere.

\subsection{The full projective stabilizer of \texorpdfstring{$\Qtwotwo$}{Q22}}

With \(H\), \(h\), and \(\Qtwotwo\) as in \eqref{eq:Q22-homogeneous}, the projective
class \([A]\in\PGL(4,\C)\) acts on \(\CP^3\) by \([A]\cdot[z]=[Az]\). We
characterize those transformations preserving \(\Qtwotwo\).

\begin{proposition}\label{prop:full-stabilizer-Q22}
Let \(A\in GL(4,\C)\).  Then the following are equivalent:
\begin{enumerate}
\item[\rm (i)] the projective transformation \([A]\in \PGL(4,\C)\) preserves \(\Qtwotwo\);
\item[\rm (ii)] there exists \(\lambda\in \R^\times\) such that
\[
A^*HA=\lambda H.
\]
\end{enumerate}
In particular, the holomorphic projective stabilizer of \(\Qtwotwo\) is
\[
\mathbb P\{A\in GL(4,\C)\mid A^*HA=\lambda H \text{ for some }\lambda\in \R^\times\},
\]
and its identity component is naturally identified with \(\PSU(2,2)\).
\end{proposition}

\begin{proof}
Assume first that \(A^*HA=\lambda H\) for some \(\lambda\in \R^\times\).  Then for every
\(z\in \C^4\setminus\{0\}\),
\[
h(Az,Az)=z^*A^*HAz=\lambda\,z^*Hz=\lambda\,h(z,z),
\]
so \([A]\) preserves the null cone of \(h\), hence \(\Qtwotwo\).
Conversely, if \([A]\) preserves \(\Qtwotwo\), then \(K:=A^*HA\) is a nondegenerate
Hermitian matrix whose projectivized null cone coincides with that of \(H\). Since two
such forms with the same null cone are proportional, \(K=\lambda H\) for some
\(\lambda\in\C^\times\), and Hermiticity forces \(\lambda\in\R^\times\).
\end{proof}

Normalizing \(A^*HA=H\), taking \(\det A=1\), and quotienting by the center gives the
connected symmetry group \(\PSU(2,2)\), the standard flat automorphism group of the
  hyperquadric. For the twistor
geometry one must additionally require compatibility with \(j\), i.e.\ that \([A]\)
preserve the family of twistor fibres. This leads to the smaller subgroup studied next.

\subsection{The involution \texorpdfstring{$j$}{j} and the subgroup commuting with it}

The twistor-compatible symmetries of $\Qtwotwo$ are those projective transformations
that commute with $j$ and hence preserve the family of twistor fibres. Recall from
Section~\ref{sec:ambient-models} that $j$ is the projectivization of $J(z)=S\overline{z}$,
with $J^2=-\Id$. The condition $[A]\circ j=j\circ[A]$ is, after rescaling, equivalent to
\begin{equation}\label{eq:AS=SbarA}
AS=S\overline{A}.
\end{equation}

\begin{proposition}\label{prop:commuting-with-j}
The element $[A]\in\PGL(4,\C)$ commutes with $j$ if and only if $A$ has the form
\begin{equation}\label{eq:quaternionic-normal-form-section}
A=
\begin{pmatrix}
a&b&c&d\\
-\overline b&\overline a&-\overline d&\overline c\\
e&f&g&h\\
-\overline f&\overline e&-\overline h&\overline g
\end{pmatrix},
\qquad a,b,c,d,e,f,g,h\in\C.
\end{equation}
\end{proposition}

\begin{proof}
The condition \eqref{eq:AS=SbarA} is equivalent to $AJ=JA$. Writing $A=(a_{rs})$
and equating the entries of $AS$ and $S\overline{A}$ yields exactly the normal form
\eqref{eq:quaternionic-normal-form-section}.
\end{proof}

The form \eqref{eq:quaternionic-normal-form-section} is the standard complex
realization of a quaternionic $2\times2$ matrix: under $\C^4\simeq\HH^2$,
$(z_0,z_1,z_2,z_3)\leftrightarrow(z_0+jz_1,z_2+jz_3)$, the matrices satisfying
\eqref{eq:AS=SbarA} are exactly the $\C$-linear maps induced by right $\HH$-linear
endomorphisms of $\HH^2$, i.e.\ elements of $GL(2,\HH)$ \cite{SalamonViaclovsky2009}.

Intersecting with the stabilizer of $\Qtwotwo$ from Proposition~\ref{prop:full-stabilizer-Q22}
gives the following.

\begin{corollary}\label{cor:j-centralizer-Q22}
The element $[A]\in\PGL(4,\C)$ preserves $\Qtwotwo$ and commutes with $j$ if and only if it
admits a representative of the form \eqref{eq:quaternionic-normal-form-section}
satisfying $A^*HA=\lambda H$ for some $\lambda\in\R^\times$. The identity component
of this subgroup is $PSp(1,1)\subset\PSU(2,2)$.
\end{corollary}

\begin{proof}
The first statement combines Propositions~\ref{prop:full-stabilizer-Q22}
and~\ref{prop:commuting-with-j}. For the second, normalizing $\lambda=1$ gives
matrices satisfying both $AS=S\overline{A}$ and $A^*HA=H$; under $\C^4\simeq\HH^2$,
the form $h$ becomes the quaternionic Hermitian form of signature $(1,1)$, preserved
by $Sp(1,1)$, whence $PSp(1,1)$ after projectivization.
\end{proof}

Geometrically, if $[A]$ commutes with $j$ then
$[A](\PP(\Span_\C\{z,Jz\}))=\PP(\Span_\C\{Az,JAz\})$,
so $[A]$ maps twistor fibres to twistor fibres. Thus $PSp(1,1)$ is exactly the group
of projective automorphisms of $\Qtwotwo$ preserving the twistor fibration
\cite{SalamonViaclovsky2009,Porter2021,Altavilla2025}.

\begin{remark}
In matrix terms, $[A]$ belongs to this subgroup if and only if it has the form
\eqref{eq:quaternionic-normal-form-section} with
\begin{align*}
\overline{a}\,c+b\,\overline{d}-\overline{e}\,g-f\,\overline{h}&=0,\\
\overline{a}\,d-b\,\overline{c}-\overline{e}\,h+f\,\overline{g}&=0,\\
|a|^2+|b|^2-|e|^2-|f|^2&=-\bigl(|c|^2+|d|^2-|g|^2-|h|^2\bigr)\neq 0.
\label{eq:explicit-jQ-condition-3}
\end{align*}
Setting $\lambda=|a|^2+|b|^2-|e|^2-|f|^2$ and defining
\[
u=\begin{pmatrix}a\\-\overline{b}\\e\\-\overline{f}\end{pmatrix},\qquad
v=\begin{pmatrix}c\\-\overline{d}\\g\\-\overline{h}\end{pmatrix},
\]
the columns of $A$ are $[\,u\ \ Ju\ \ v\ \ Jv\,]$ and the conditions above become
$h(u,v)=0$, $h(u,Jv)=0$, $h(u,u)=-h(v,v)\neq 0$: the subgroup consists exactly of
the projectivized quaternionic-linear transformations preserving $h$ up to a real
nonzero factor.
\end{remark}

\begin{remark}
Any $[A]$ commuting with $j$ preserves twistor fibres and therefore descends to a
transformation of $S^4$. If it also preserves $\Qtwotwo=\pi^{-1}(S^3)$, the induced
transformation preserves $S^3\subset S^4$. Thus $PSp(1,1)$ may be viewed as the
twistor lift of the conformal symmetry group of the pair $(S^4,S^3)$.
\end{remark}
\section{Holomorphic curves contained in \texorpdfstring{$\Qtwotwo$}{Q22}}\label{sec:curves}
Holomorphic curves in $\Qtwotwo$ are closely related to the null directions of the
Levi form and to its split signature $(1,1)$; in the
projective model they also admit a concrete interpretation in terms of lines in $\CP^3$
and the twistor geometry of $S^3$.

We shall show that $\Qtwotwo$ contains exactly two families of projective lines, i.e.  the
twistor fibres, parametrized by points of $S^3$, and a complementary family of
transverse lines projecting to round $2$-spheres in the base.

\subsection{The two rulings by projective lines}
We now describe the projective lines contained in $\Qtwotwo$. \par

It is convenient to use the block decomposition
\[
\C^4=\C^2_+\oplus\C^2_-,
\qquad
[z_0:z_1:z_2:z_3]=[z:w],
\qquad
z=\begin{pmatrix}z_0\\z_1\end{pmatrix},\quad
w=\begin{pmatrix}z_2\\z_3\end{pmatrix}.
\]
In these coordinates $h((z,w),(z,w))=\|z\|^2-\|w\|^2$, so that
\[
\Qtwotwo=\bigl\{[z:w]\in\PP(\C^2_+\oplus\C^2_-)\;\big|\;\|z\|^2=\|w\|^2\bigr\}.
\]
For $A\in U(2)$, set
\[
L_A:=\{(z,Az)\mid z\in\C^2\}\subset\C^2_+\oplus\C^2_-,
\qquad
\ell_A:=\PP(L_A)\subset\CP^3.
\]
Since $A$ is unitary, every vector $(z,Az)$ is $h$-null, hence $\ell_A\subset\Qtwotwo$.
We now show that every projective line in $\Qtwotwo$ arises this way.

\begin{theorem}\label{thm:all-lines-Q22}
A projective line \(\ell\subset \CP^3\) is contained in \(\Qtwotwo\) if and only if
\[
\ell=\ell_A
\qquad\text{for a unique }A\in U(2).
\]
Equivalently, the projective lines contained in \(\Qtwotwo\) are in one-to-one
correspondence with the unitary group \(U(2)\).
\end{theorem}

\begin{proof}
Let \(\ell=\PP(L)\subset \Qtwotwo\), where \(L\subset \C^2_+\oplus \C^2_-\) is a
complex \(2\)-plane.  Since \(\ell\subset \Qtwotwo\), every vector of \(L\) is
\(h\)-null, so
\[
h|_L\equiv 0.
\]
Consider the projection
\[
p_+:L\longrightarrow \C^2_+,
\qquad
(z,w)\longmapsto z.
\]
We claim that \(p_+\) is injective.  Indeed, if \((0,w)\in L\), then
\[
0=h((0,w),(0,w))=-\|w\|^2,
\]
hence \(w=0\).  Since \(\dim_{\C}L=2=\dim_{\C}\C^2_+\), the map \(p_+\) is an
isomorphism.  Therefore \(L\) is the graph of a unique complex linear map
\[
A:\C^2_+\longrightarrow \C^2_-,
\qquad
L=L_A=\{(z,Az)\mid z\in \C^2\}.
\]
The nullity condition now gives, for every \(z\in \C^2\),
\[
0=h((z,Az),(z,Az))=\|z\|^2-\|Az\|^2,
\]
so \(\|Az\|=\|z\|\) for all \(z\), that is, \(A\in U(2)\).

Conversely, if \(A\in U(2)\), then
\[
h((z,Az),(z,Az))=\|z\|^2-\|Az\|^2=0
\qquad
\forall\,z\in \C^2,
\]
hence every point of \(\ell_A\) lies on \(\Qtwotwo\).  Uniqueness is immediate from the
fact that the graph of a linear map is uniquely determined by the map itself.
\end{proof}

We now determine which of these lines are twistor fibres.  Let
\[
J_0=
\begin{pmatrix}
0&-1\\
1&0
\end{pmatrix},
\]
so that the quaternionic anti-linear map \(J\) on \(\C^4=\C^2\oplus \C^2\) is given by
\[
J(z,w)=(J_0\overline z,\;J_0\overline w),
\]
and its projectivization is precisely the involution \(j\).  As observed earlier, the
twistor fibres are exactly the \(j\)-invariant projective lines.  This translates into a
simple condition on the matrix \(A\).

\begin{theorem}\label{thm:fibres-vs-lines}
Let \(A\in U(2)\). Then the following conditions are equivalent:
\begin{enumerate}
\item[\rm(i)] \(\ell_A\) is a twistor fibre;
\item[\rm(ii)] \(AJ_0=J_0\overline{A}\);
\item[\rm(iii)] \(A\in SU(2)\).
\end{enumerate}
\end{theorem}

\begin{proof}
We already know that
$\ell_A$ is a twistor fibre if and only if it is $j$-invariant, i.e.\ if and only if
$L_A$ is $J$-stable. Since $L_A=\{(z,Az)\mid z\in\C^2\}$, $J$-stability requires
$J(z,Az)=(J_0\overline{z},J_0\overline{A}\,\overline{z})\in L_A$ for all $z$, which
gives $AJ_0\overline{z}=J_0\overline{A}\,\overline{z}$ for all $z$, i.e.\ condition
(ii). This proves (i)$\iff$(ii).

For (ii)$\iff$(iii), write $A=\bigl(\begin{smallmatrix}\alpha&\beta\\\gamma&\delta
\end{smallmatrix}\bigr)\in U(2)$. Then $AJ_0=J_0\overline{A}$ forces
$\gamma=-\overline{\beta}$ and $\delta=\overline{\alpha}$, so
\[
A=\begin{pmatrix}\alpha&\beta\\-\overline{\beta}&\overline{\alpha}\end{pmatrix}.
\]
Unitarity gives $|\alpha|^2+|\beta|^2=1$, hence $\det A=1$ and $A\in SU(2)$.
Conversely, every $A\in SU(2)$ has this form and satisfies $AJ_0=J_0\overline{A}$.
\end{proof}

The theorem shows that the twistor fibres form the distinguished subfamily
\[
\{\ell_A\mid A\in SU(2)\}\subset \{\ell_A\mid A\in U(2)\}.
\]
The remaining lines are the ones that are not preserved by \(j\).  Since every unitary
matrix admits a decomposition
\[
A=e^{i\theta}U,
\qquad
U\in SU(2),
\qquad
\theta\in \R/2\pi\Z,
\]
this second family is measured precisely by the phase \(e^{i\theta}\).

\begin{corollary}\label{cor:two-line-families}
The projective lines contained in \(\Qtwotwo\) split into two twistorially distinct
types:
\begin{enumerate}
\item[\rm (a)] the \(j\)-invariant lines, equivalently the twistor fibres, corresponding
to \(A\in SU(2)\);
\item[\rm (b)] the non-\(j\)-invariant lines, corresponding to matrices
\(A=e^{i\theta}U\in U(2)\) with \(U\in SU(2)\) and \(e^{i\theta}\neq \pm1\).
\end{enumerate}
Moreover, if \(A=e^{i\theta}U\), then
\[
j(\ell_A)=\ell_{e^{-i\theta}U}.
\]
In particular, every non-fibre line occurs in a \(j\)-conjugate pair.
\end{corollary}

\begin{proof}
Only the last statement remains to be checked.  Let \(A=e^{i\theta}U\) with
\(U\in SU(2)\).  Using \(UJ_0=J_0\overline U\), we compute
\[
J(z,Az)
=
(J_0\overline z,\;J_0\overline{Az})
=
(J_0\overline z,\;e^{-i\theta}J_0\overline U\,\overline z)
=
(J_0\overline z,\;e^{-i\theta}U J_0\overline z).
\]
Thus \(J(L_A)=L_{e^{-i\theta}U}\), and therefore
\[
j(\ell_A)=\ell_{e^{-i\theta}U}.
\]
If \(e^{i\theta}\neq \pm1\), then \(\ell_A\) is not \(j\)-invariant by
Theorem~\ref{thm:fibres-vs-lines}, so the two lines are distinct.
\end{proof}

The theorem shows that the twistor fibres form the distinguished subfamily
\[
\{\ell_A\mid A\in SU(2)\}\subset \{\ell_A\mid A\in U(2)\}.
\]
The remaining lines are those not preserved by $j$. Every unitary matrix admits a
unique decomposition $A=e^{i\theta}U$ with $U\in SU(2)$ and $\theta\in(0,\pi)$
after choosing the representative angle in $(0,\pi)$ (the cases $\theta=0$ and
$\theta=\pi$ give $A\in SU(2)$ and $A\in-SU(2)$, both of which yield twistor
fibres by Theorem~\ref{thm:fibres-vs-lines}). The second family of lines is
therefore parametrized by $\theta\in(0,\pi)$ together with $U\in SU(2)$, and the
phase $e^{i\theta}$ is the complete twistorial invariant distinguishing them from
the fibres.

\subsection{Incidence, intersections, and sphere geometry}

We now relate the incidence geometry of the lines \(\ell_A\subset \Qtwotwo\) to the
conformal geometry of the base \(S^3\).  The guiding idea is very much in the spirit of
Shapiro's twistor interpretation of sphere geometry: a line upstairs should be read
downstairs through the way it meets the family of twistor fibres \cite{Shapiro2013}.
In our setting this leads to a particularly simple description, because the fibres are
already parametrized by \(SU(2)\cong S^3\).

Using Theorem~\ref{thm:fibres-vs-lines}, we identify the base \(S^3\) with the family of
twistor fibres
\[
\{\ell_x\mid x\in SU(2)\},
\]
where \(\ell_x\) denotes the fibre corresponding to the matrix \(x\in SU(2)\).  Let
\[
A=e^{i\theta}U\in U(2),
\qquad
U\in SU(2),
\qquad
\theta\in \R/2\pi\Z.
\]
We wish to describe the image \(\pi(\ell_A)\subset S^3\).

\begin{theorem}\label{thm:line-to-sphere}
Let \(A=e^{i\theta}U\in U(2)\), with \(U\in SU(2)\).  Then
\begin{equation}\label{eq:SigmaA-def}
\pi(\ell_A)
=
\Sigma_A
:=
\{x\in SU(2)\mid \det(A-x)=0\}.
\end{equation}
Equivalently,
\begin{equation}\label{eq:SigmaA-conjugacy}
\Sigma_A
=
U\,C_\theta,
\qquad
C_\theta
:=
\{y\in SU(2)\mid tr(y)=2\cos\theta\}.
\end{equation}
In particular:
\begin{enumerate}
\item[\rm (i)] if \(A\in SU(2)\), then \(\Sigma_A=\{U\}\), so \(\ell_A\) is a twistor
fibre;
\item[\rm (ii)] if \(A\notin SU(2)\), then \(\Sigma_A\) is a round \(2\)-sphere in
\(S^3\).
\end{enumerate}
\end{theorem}

\begin{proof}
A point \(x\in SU(2)\cong S^3\) belongs to \(\pi(\ell_A)\) if and only if the fibre
\(\ell_x\) meets the line \(\ell_A\). By Theorem~\ref{thm:all-lines-Q22}, these two
lines are the projectivizations of the graphs
\[
L_A=\{(z,Az)\mid z\in\C^2\},
\qquad
L_x=\{(z,xz)\mid z\in\C^2\}.
\]
Hence
\[
\ell_A\cap \ell_x\neq\varnothing
\]
if and only if there exists \(z\neq 0\) such that
\[
(z,Az)=(z,xz),
\]
that is, such that
\[
(A-x)z=0.
\]
Equivalently, \(\ell_A\) and \(\ell_x\) intersect if and only if \(A-x\) is singular,
namely
\[
\det(A-x)=0.
\]
This proves \eqref{eq:SigmaA-def}.

Now write \(A=e^{i\theta}U\).  Then
\[
\det(A-x)=0
\quad\Longleftrightarrow\quad
\det(e^{i\theta}I-U^{-1}x)=0.
\]
Set \(y=U^{-1}x\in SU(2)\).  Since \(\det y=1\), we compute
\[
0=\det(e^{i\theta}I-y)
=
e^{2i\theta}-e^{i\theta}tr(y)+1.
\]
Multiplying by \(e^{-i\theta}\) gives
\[
0=2\cos\theta-tr(y).
\]
Hence \(\det(A-x)=0\) if and only if \(tr(U^{-1}x)=2\cos\theta\), which is exactly
\eqref{eq:SigmaA-conjugacy}.

If \(A\in SU(2)\), then \(\theta=0\) modulo \(\pi\), and by
Theorem~\ref{thm:fibres-vs-lines} the line \(\ell_A\) is a fibre; therefore its image is a
single point.  If \(A\notin SU(2)\), then \(0<\theta<\pi\) after choosing a representative
angle, and the condition \(tr(y)=2\cos\theta\) cuts out a standard round \(2\)-sphere
inside the unit \(3\)-sphere \(SU(2)\).
\end{proof}

The theorem shows that every non-fibre line in \(\Qtwotwo\) projects to a round
\(2\)-sphere in \(S^3\).  Moreover, the involution \(j\) exchanges the two oriented lifts of
the same sphere.

\begin{corollary}
Let \(A=e^{i\theta}U\in U(2)\), with \(U\in SU(2)\).  Then
\[
\pi(\ell_A)=\pi(j(\ell_A))=\pi(\ell_{e^{-i\theta}U}).
\]
Thus the two \(j\)-conjugate lines \(\ell_{e^{i\theta}U}\) and \(\ell_{e^{-i\theta}U}\)
project to the same underlying round \(2\)-sphere in \(S^3\).  They may be regarded as
the two opposite twistor lifts of that sphere.
\end{corollary}

\begin{proof}
By Corollary~\ref{cor:two-line-families},
\(j(\ell_{e^{i\theta}U})=\ell_{e^{-i\theta}U}\).
On the other hand,
\[
\Sigma_{e^{i\theta}U}
=
U\,C_\theta
=
U\,C_{-\theta}
=
\Sigma_{e^{-i\theta}U},
\]
because \(\cos(-\theta)=\cos\theta\).
\end{proof}

We next compare intersections of lines upstairs with intersections of the associated
spheres downstairs.  For
\[
A=e^{i\theta}U,
\qquad
B=e^{i\phi}V,
\qquad
U,V\in SU(2),
\]
let \(\Sigma_A,\Sigma_B\subset S^3\) be the corresponding round \(2\)-spheres.  Since
\(SU(2)\) is the unit \(3\)-sphere with its round metric, the subsets
\[
C_\theta=\{y\in SU(2)\mid tr(y)=2\cos\theta\}
\]
are precisely the geodesic \(2\)-spheres of radius \(\theta\) centered at \(\pi(\ell_U)\).
Thus \(\Sigma_A=U C_\theta\) is the geodesic \(2\)-sphere of radius \(\theta\) centered at
\(U\).

\begin{theorem}\label{thm:incidence-sphere-geometry}
Let
\[
A=e^{i\theta}U,
\qquad
B=e^{i\phi}V,
\qquad
U,V\in SU(2),
\]
with \(0\leq \theta,\phi\leq \pi\).  Then the following hold.
\begin{enumerate}
\item[\rm (i)] A twistor fibre \(\ell_x\) intersects \(\ell_A\) if and only if
\(x\in \Sigma_A\).

\item[\rm (ii)] The lines \(\ell_A\) and \(\ell_B\) intersect if and only if
\begin{equation}\label{eq:line-intersection-trace}
tr(V^{-1}U)=2\cos(\theta-\phi).
\end{equation}
Equivalently, the associated oriented spheres \(\Sigma_A\) and \(\Sigma_B\) are tangent
with compatible orientation.

\item[\rm (iii)] The lines \(\ell_A\) and \(j(\ell_B)=\ell_{e^{-i\phi}V}\) intersect if and
only if
\[
tr(V^{-1}U)=2\cos(\theta+\phi).
\]
Equivalently, the underlying spheres \(\Sigma_A\) and \(\Sigma_B\) are tangent with
opposite orientation.
\end{enumerate}
\end{theorem}

\begin{proof}
Statement \((i)\) is just Theorem~\ref{thm:line-to-sphere}: by definition,
\[
x\in \Sigma_A
\iff
\det(A-x)=0
\iff
\ell_A\cap \ell_x\neq \varnothing.
\]

For \((ii)\), Theorem~\ref{thm:all-lines-Q22} gives
\[
\ell_A\cap \ell_B\neq \varnothing
\iff
\det(A-B)=0.
\]
Substituting \(A=e^{i\theta}U\) and \(B=e^{i\phi}V\), this becomes
\[
0=\det(e^{i\theta}U-e^{i\phi}V)
=
\det\bigl(e^{i(\theta-\phi)}V^{-1}U-I\bigr).
\]
Since \(V^{-1}U\in SU(2)\), the determinant vanishes if and only if \(1\) is an
eigenvalue of \(e^{i(\theta-\phi)}V^{-1}U\), equivalently if and only if
\(V^{-1}U\) has eigenvalues \(e^{\pm i(\theta-\phi)}\).  This is equivalent to
\eqref{eq:line-intersection-trace}.  Since \(\Sigma_A\) and \(\Sigma_B\) are geodesic
\(2\)-spheres of radii \(\theta\) and \(\phi\), centered at \(\pi(\ell_U)\) and \(\pi(\ell_V)\), the resulting condition
\eqref{eq:line-intersection-trace} is exactly the tangency condition with compatible
orientation.

For \((iii)\), replace \(B\) by \(e^{-i\phi}V\).  Then
\[
\ell_A\cap j(\ell_B)\neq \varnothing
\iff
\det(A-e^{-i\phi}V)=0,
\]
and the same computation yields
\(tr(V^{-1}U)=2\cos(\theta+\phi)\).
This is the tangency condition with the opposite choice of orientation on the second
sphere.
\end{proof}

\begin{example}
The correspondence described above is already visible in a few very simple cases.

\begin{enumerate}
\item[\rm(i)] \textbf{A twistor fibre.}
Take \(A=I\in SU(2)\). Then \(\ell_A\) is a twistor fibre, and
\[
\Sigma_A
=
\{x\in SU(2)\mid \det(I-x)=0\}
=
\{I\}.
\]
Indeed, if \(x\in SU(2)\) and \(\det(I-x)=0\), then \(1\) is an eigenvalue of \(x\);
since \(\det x=1\), both eigenvalues are equal to \(1\), hence \(x=I\).

\item[\rm(ii)] \textbf{A standard round \(2\)-sphere.}
Take
\[
A=e^{i\theta}I,
\qquad 0<\theta<\pi.
\]
Then
\[
\Sigma_A
=
\{x\in SU(2)\mid \det(e^{i\theta}I-x)=0\}
=
\{x\in SU(2)\mid tr(x)=2\cos\theta\}
=
C_\theta.
\]
Thus \(\ell_A\) projects to the geodesic \(2\)-sphere of radius \(\theta\) centered at
the identity. In particular, for \(\theta=\pi/2\),
\[
\Sigma_{iI}
=
\{x\in SU(2)\mid tr(x)=0\},
\]
which is the equatorial round \(2\)-sphere in \(S^3\).

\item[\rm(iii)] \textbf{Compatible tangency.}
Let
\[
A=e^{i\pi/3}I,
\qquad
B=e^{i\pi/6}
\begin{pmatrix}
e^{i\pi/6}&0\\
0&e^{-i\pi/6}
\end{pmatrix}.
\]
Then \(A=e^{i\theta}U\) and \(B=e^{i\phi}V\) with
\[
\theta=\frac{\pi}{3},\qquad \phi=\frac{\pi}{6},\qquad U=I,
\qquad
V=
\begin{pmatrix}
e^{i\pi/6}&0\\
0&e^{-i\pi/6}
\end{pmatrix}\in SU(2).
\]
Moreover,
\[
tr(V^{-1}U)=tr(V^{-1})=2\cos\frac{\pi}{6}
=
2\cos\!\left(\frac{\pi}{3}-\frac{\pi}{6}\right).
\]
Hence, by Theorem~\ref{thm:incidence-sphere-geometry}\rm(ii),
\[
\ell_A\cap \ell_B\neq\varnothing,
\]
so the corresponding spheres \(\Sigma_A\) and \(\Sigma_B\) are tangent with
compatible orientation.

\item[\rm(iv)] \textbf{Opposite tangency.}
Let
\[
A=e^{i\pi/4}I,
\qquad
B=e^{i\pi/4}
\begin{pmatrix}
i&0\\
0&-i
\end{pmatrix}.
\]
Then \(A=e^{i\theta}U\) and \(B=e^{i\phi}V\) with
\[
\theta=\phi=\frac{\pi}{4},\qquad U=I,
\qquad
V=
\begin{pmatrix}
i&0\\
0&-i
\end{pmatrix}\in SU(2).
\]
Since
\[
tr(V^{-1}U)=tr(V^{-1})=0
=
2\cos\frac{\pi}{2}
=
2\cos\!\left(\frac{\pi}{4}+\frac{\pi}{4}\right),
\]
Theorem~\ref{thm:incidence-sphere-geometry}\rm(iii) gives
\[
\ell_A\cap j(\ell_B)\neq\varnothing.
\]
Therefore \(\Sigma_A\) and \(\Sigma_B\) are tangent with opposite orientation.
\end{enumerate}
\end{example}

\subsection{No other holomorphic curves in \texorpdfstring{$\Qtwotwo$}{Q22}}

Every connected holomorphic curve in $\Qtwotwo$ is contained in one of the
projective lines classified above; there are no curves of higher degree.

\begin{theorem}
Let $C\subset\Qtwotwo$ be a connected complex curve. Then $C$ is contained in a
projective line $\ell_A\subset\Qtwotwo$, with $A\in U(2)$.
\end{theorem}

\begin{proof}
Choose a holomorphic lift $F:\Delta\to\C^4\setminus\{0\}$ of $C$ near any point.
The condition $[F(\zeta)]\in\Qtwotwo$ is $h(F,F)=0$. Differentiating twice and
using holomorphicity, one obtains $h(F',F)=0$ and $h(F',F')=0$, so the plane
$L(\zeta):=\Span_\C\{F,F'\}$ is totally isotropic whenever $F\wedge F'\neq 0$.
Since $h$ has signature $(2,2)$, a $2$-dimensional totally isotropic subspace
satisfies $L(\zeta)^\perp=L(\zeta)$; differentiating once more gives
$F''\in L(\zeta)^\perp=L(\zeta)$, so $L(\zeta)$ is locally constant. Hence $C$
is locally contained in the fixed projective line $\PP(L)$, and by connectedness
the whole curve lies in a single line $\ell\subset\Qtwotwo$.
\end{proof}

\begin{corollary}
$\Qtwotwo$ contains no irreducible complex projective curves of degree $>1$. The
connected holomorphic curves are exactly the open subsets of the lines $\ell_A$.
\end{corollary}

This is consistent with Bryant's general theory \cite{Bryant1982},  in a Lorentzian \(CR\) \(5\)-manifold, i.e.\ a \(CR\) manifold whose Levi form has
signature \((n-1,1)\) with \(n=\mathrm{rk}_{\C}T^{0,1}M\), holomorphic curves depend on at most \(4\) real parameters
\cite[Thm.~3.11]{Bryant1982}, and the flat model $\Qtwotwo$ is precisely the one
that attains this maximum \cite[Thm.~4.1]{Bryant1982}.
\begin{remark}
The preceding theorem fits naturally into LeBrun's classification of foliated
Levi--nondegenerate \(CR\) manifolds \cite{LeBrun1985}. Indeed, LeBrun shows
that if a \(CR\) \((4m+1)\)-manifold with nondegenerate Levi form is foliated by
compact complex \(m\)-folds, when \(m=1\), then the leaves are
\(\CP^1\)'s, and the manifold arises from a twistor construction over a
\(3\)-manifold. More precisely, such a foliated \(CR\) manifold is determined by
its first and second fundamental forms, and conversely these data determine the
foliated \(CR\) structure.

Accordingly, \(\Qtwotwo\) should be viewed as the flat projective model for this
twistor class in dimension five. The fact established above, namely that every
connected complex curve in \(\Qtwotwo\) is contained in a projective line,
expresses in the flat model the general principle that the relevant compact
complex curves are precisely the twistor lines.
\end{remark}

\section{Hyperplane sections of \texorpdfstring{$\Qtwotwo$}{Q22}}\label{sec:hyperplane-sections}
We now study $\Qtwotwo$ by intersecting it with algebraic hypersurfaces in $\CP^3$.
From the twistor point of view this is natural: the behaviour of an algebraic surface
in $\CP^3$ with respect to the twistor fibration encodes geometric information on $S^4$,
as illustrated by the theory of real quadrics and their discriminant loci developed in
\cite{SalamonViaclovsky2009}. Our aim is to adapt this philosophy to the flat twistor
$CR$ geometry of $\Qtwotwo$.

The simplest case is hyperplane sections, which already exhibit the main phenomena, 
smooth spherical sections recovering the standard spherical $CR$ $3$-sphere, tangent
degenerations leading to Levi-flat behaviour, and the distinction between
$j$-compatible and $j$-incompatible slices. Quadric sections, treated in the next
section, form the natural continuation, $j$-invariant quadrics in $\CP^3$ have a
twistor meaning of their own, and their intersections with $\Qtwotwo$ produce branch
loci and discriminant phenomena readable on $S^3$. Together the two sections form
a single projective-twistor investigation, first at the linear and then at the quadratic
level.

\subsection{A first direct classification}

Let $\Pi=\PP(W)\subset\CP^3$ be a complex hyperplane, with $W\subset\C^4$ a
$3$-dimensional subspace. Then, from Formula \eqref{eq:Q22-homogeneous}
\[
\Pi\cap\Qtwotwo=\{[z]\in\PP(W)\mid h|_W(z,z)=0\},
\]
so the geometry of the section is governed entirely by the restriction of $h$
to $W$.

\begin{theorem}\label{thm:first-hyperplane-classification}
Let $\Pi=\PP(W)\subset\CP^3$ be a complex hyperplane. Since $h$ is nondegenerate
on $\C^4$ and $\dim_\C W=3$, the restriction $h|_W$ is either nondegenerate or has
a one-dimensional radical
\[
\mathrm{rad}(h|_W):=\{v\in W\mid h(v,w)=0\ \forall\,w\in W\}=W\cap W^{\perp_h}.
\]
Exactly two cases occur.

\begin{enumerate}
\item[\rm(i)] \textbf{Smooth case.} If $h|_W$ is nondegenerate, it has signature
$(2,1)$ or $(1,2)$, and in suitable coordinates on $\Pi\cong\CP^2$,
\begin{equation}\label{eq:smooth-section-model}
\Pi\cap\Qtwotwo
=\Bigl\{[\zeta_0:\zeta_1:\zeta_2]\in\CP^2\;\Big|\;
|\zeta_0|^2+|\zeta_1|^2=|\zeta_2|^2\Bigr\}.
\end{equation}
This is a smooth spherical $CR$ $3$-sphere, the standard projective model of
$\{|z_1|^2+|z_2|^2=1\}\subset\C^2$.

\item[\rm(ii)] \textbf{Tangent case.} If $h|_W$ is degenerate, the induced form on
$W/\mathrm{rad}(h|_W)$ is nondegenerate of signature $(1,1)$. In coordinates adapted
to the radical,
\begin{equation}\label{eq:tangent-section-model}
\Pi\cap\Qtwotwo
=\Bigl\{[\zeta_0:\zeta_1:\zeta_2]\in\CP^2\;\Big|\;
|\zeta_1|^2=|\zeta_2|^2\Bigr\},
\end{equation}
which is singular at $[1:0:0]$. On the affine charts $\zeta_1=1$ and $\zeta_2=1$ the smooth locus is
$\C\times S^1$, and the induced $CR$ structure is Levi-flat.
\end{enumerate}
\end{theorem}

\begin{proof}
Since \(h\) is nondegenerate on \(\C^4\) and \(\dim_{\C}W=3\), the orthogonal complement
\(W^{\perp_h}\) is one-dimensional.  Moreover,
\[
\rad(h|_W)=W\cap W^{\perp_h},
\]
so \(\dim_{\C}\rad(h|_W)\leq 1\).  Hence \(h|_W\) is either nondegenerate or has corank
one.

If \(h|_W\) is nondegenerate, then its signature on the \(3\)-dimensional complex space
\(W\) must be either \((2,1)\) or \((1,2)\), since the ambient signature is \((2,2)\).
By the Hermitian version of Sylvester's law of inertia, after a linear change of
coordinates on \(W\) one may write
\[
h|_W
=
|\zeta_0|^2+|\zeta_1|^2-|\zeta_2|^2
\]
or its negative.  The null cone is therefore given by
\eqref{eq:smooth-section-model}.  This is the standard spherical \(CR\) \(3\)-sphere in
\(\CP^2\), equivalently the projective model of the usual sphere
\(\{|z_1|^2+|z_2|^2=1\}\subset \C^2\) (see Setion~\ref{3sph}).

If \(h|_W\) is degenerate, then \(\rad(h|_W)\) is a complex line.  The induced form on
the quotient \(W/\rad(h|_W)\) is nondegenerate of signature \((1,1)\).  Choosing
coordinates adapted to the radical, one obtains
\[
h|_W
=
|\zeta_1|^2-|\zeta_2|^2,
\]
where \(\zeta_0\) is the coordinate along the radical direction.  Hence the section is
given by \eqref{eq:tangent-section-model}.  
It is singular precisely at the point \([1:0:0]\), where both \(\zeta_1\) and \(\zeta_2\) vanish. 
Away from this point, the equation \(|\zeta_1|=|\zeta_2|\) implies that there exists a unique
\(\vartheta\in\R/2\pi\Z\) such that
\[
\zeta_2=e^{i\vartheta}\zeta_1.
\]
Hence the smooth locus can be written as
\[
\bigcup_{\vartheta\in\R/2\pi\Z}
\Bigl\{
[\zeta_0:\zeta_1:e^{i\vartheta}\zeta_1]\in\CP^2
\;\Big|\;
\zeta_1\neq 0
\Bigr\}.
\]
For each fixed \(\vartheta\), the set
\[
L_\vartheta:=
\Bigl\{
[\zeta_0:\zeta_1:e^{i\vartheta}\zeta_1]\in\CP^2
\Bigr\}
\]
is a projective line contained in \(\Pi\cap\Qtwotwo\), and the smooth locus is foliated by the
complex curves \(L_\vartheta\setminus\{[1:0:0]\}\).
On the affine chart \(\zeta_2=1\) (and similarly on \(\zeta_1=1)\)), writing
\[
u=\frac{\zeta_0}{\zeta_2},
\qquad
v=\frac{\zeta_1}{\zeta_2},
\]
the section becomes
\[
\{(u,v)\in\C^2\mid |v|=1\}\cong \C\times S^1.
\]
For fixed \(\vartheta\), the leaf \(L_\vartheta\setminus\{[1:0:0]\}\) is given by
\[
v=e^{-i\vartheta},
\qquad
u\in\C,
\]
hence is a complex affine line. Therefore the smooth locus is foliated by complex
one-dimensional submanifolds, and since it is locally a direct product, its induced \(CR\) structure is Levi-flat.
\end{proof}

The smooth case of Theorem~\ref{thm:first-hyperplane-classification} recovers exactly
the spherical $CR$ $3$-sphere of Section~\ref{sec:embedded-spheres}: the section
$\Pi\cap\Qtwotwo$ is projectively equivalent to
$\{[\zeta_0:\zeta_1:\zeta_2]\in\CP^2\mid|\zeta_0|^2=|\zeta_1|^2+|\zeta_2|^2\}$,
now appearing as a projective slice of the flat twistor model rather than as a
hypersurface of $\CP^2$.

This recovery is especially transparent when the hyperplane contains a twistor fibre.

\begin{corollary}\label{prop:hyperplane-containing-fibre}
Let $\ell_p=\pi^{-1}(p)$ be the twistor fibre over $p\in S^4$ and let
$\Pi\supset\ell_p$ be a hyperplane.
\begin{enumerate}
\item[\rm(i)] If $p\notin S^3$, then $\ell_p\cap\Qtwotwo=\varnothing$, the section
$\Pi\cap\Qtwotwo$ is a smooth spherical $CR$ $3$-sphere, and in the affine chart
$\Pi\setminus\ell_p\cong\C^2$ it is modeled by
$\{|z_1|^2+|z_2|^2=1\}\subset\C^2$.
\item[\rm(ii)] If $p\in S^3$, then $\ell_p\subset\Qtwotwo$ and $\Pi\cap\Qtwotwo$
is of tangent type with Levi-flat smooth locus.
\end{enumerate}
\end{corollary}

\begin{proof}
Since $\Qtwotwo=\pi^{-1}(S^3)$, the fibre $\ell_p$ meets $\Qtwotwo$ if and only if
$p\in S^3$. In case (i), $\ell_p$ is disjoint from the section, so
Theorem~\ref{thm:first-hyperplane-classification} forces the smooth case; taking
$\ell_p$ as the line at infinity in $\Pi\cong\CP^2$ gives the standard affine sphere.
In case (ii), $\ell_p\subset\Pi\cap\Qtwotwo$, so the section contains a complex line
and cannot be Levi-nondegenerate; the only remaining case in the theorem is the
tangent one.
\end{proof}

Thus a hyperplane section of $\Qtwotwo$ recovers the spherical $CR$ geometry of
Section~\ref{sec:embedded-spheres} precisely when it contains a twistor fibre over a
point outside $S^3$.

\subsection{Classification up to \texorpdfstring{$j$}{j}-compatible symmetries}
We now revisit the hyperplane classification from the point of view of the
twistor-compatible symmetry group $\mathcal{G}_{j,\Qtwotwo}\simeq PSp(1,1)$ of
Corollary~\ref{cor:j-centralizer-Q22}.

A hyperplane $\Pi_v=\{[z]\in\CP^3\mid v_0z_0+v_1z_1+v_2z_2+v_3z_3=0\}$ has
$h$-normal vector
\[
n_v=(\overline{v_0},\overline{v_1},-\overline{v_2},-\overline{v_3})^T,
\]
so that $\Pi_v=(n_v)^{\perp_h}$. The key invariant is
\[
\Delta(\Pi_v):=h(n_v,n_v)=|v_0|^2+|v_1|^2-|v_2|^2-|v_3|^2.
\]

\begin{theorem}\label{thm:hyperplanes-Gj}
Under the action of \(\mathcal{G}_{j,\Qtwotwo}\), the orbit of a complex hyperplane \(\Pi\subset \CP^3\) is
completely determined by the sign of \(\Delta(\Pi)\).  More precisely, \(\mathcal{G}_{j,\Qtwotwo}\) has
exactly three orbits on the set of hyperplanes, represented by
\[
\Pi_+:=\{z_0=0\},
\qquad
\Pi_-:=\{z_2=0\},
\qquad
\Pi_0:=\{z_0-z_2=0\},
\]
corresponding respectively to
\[
\Delta(\Pi_+)>0,
\qquad
\Delta(\Pi_-)<0,
\qquad
\Delta(\Pi_0)=0.
\]
\end{theorem}

\begin{proof}
If \([A]\in \mathcal{G}_{j,\Qtwotwo}\), then \(A^*HA=H\), so
\[
h(An,An)=h(n,n)
\qquad\text{for all }n\in \C^4.
\]
Hence the sign of \(\Delta(\Pi)=h(n_\Pi,n_\Pi)\) is preserved under the action of
\(\mathcal{G}_{j,\Qtwotwo}\).  It remains to show that this sign is also a complete invariant.

Write
\(n=(x,y)\), for
\(x,y\in \C^2\),
so that
\[
h(n,n)=\|x\|^2-\|y\|^2.
\]
For a nonzero vector \(w=(\xi,\eta)\in \C^2\), define
\[
U_w
=
\frac{1}{\sqrt{|\xi|^2+|\eta|^2}}
\begin{pmatrix}
\overline{\xi}&\overline{\eta}\\
-\eta&\xi
\end{pmatrix}.
\]
Then \(U_w\in U(2)\), \(U_ww=(\|w\|,0)\), and \(U_w\) has quaternionic form, hence
commutes with \(J_0\).  Therefore the block-diagonal matrix
\[
D_n
=
\begin{pmatrix}
U_x&0\\
0&U_y
\end{pmatrix}
\]
satisfies
\[
D_n^*HD_n=H,
\qquad
D_nS=S\overline{D_n},
\]
so \([D_n]\in \mathcal{G}_{j,\Qtwotwo}\).

Assume first that \(h(n,n)>0\).  After rescaling \(n\), we may assume
\[
h(n,n)=1.
\]
Set
\(r:=\|x\|,
s:=\|y\|\),
so that \(r^2-s^2=1\).  Applying \(D_n\), we may suppose
\[
n=(r,0,s,0)^T.
\]
Now consider
\[
C_{r,s}
=
\begin{pmatrix}
rI_2&sI_2\\
sI_2&rI_2
\end{pmatrix}.
\]
Since \(r,s\in \R\) and \(r^2-s^2=1\), one has
\[
C_{r,s}^*HC_{r,s}=H,
\qquad
C_{r,s}S=S\overline{C_{r,s}},
\]
so \([C_{r,s}]\in \mathcal{G}_{j,\Qtwotwo}\).  Moreover,
\[
C_{r,s}e_0=(r,0,s,0)^T.
\]
Thus every positive vector is \(\mathcal{G}_{j,\Qtwotwo}\)-equivalent to \(e_0\), and hence every positive
hyperplane is \(\mathcal{G}_{j,\Qtwotwo}\)-equivalent to
\[
\Pi_+=e_0^{\perp_h}=\{z_0=0\}.
\]

The negative case is entirely analogous.  If \(h(n,n)<0\), normalize so that
\(h(n,n)=-1\), write
\(r:=\|y\|,
s:=\|x\|,
r^2-s^2=1\),
and after applying \(D_n\) reduce to \(n=(s,0,r,0)^T\).  Since
\(C_{r,s}e_2=(s,0,r,0)^T\),
every negative vector is \(\mathcal{G}_{j,\Qtwotwo}\)-equivalent to \(e_2\), so every negative hyperplane is
equivalent to
\[
\Pi_-=e_2^{\perp_h}=\{z_2=0\}.
\]

Finally, assume \(h(n,n)=0\).  Then \(\|x\|=\|y\|=:r>0\).  Applying \(D_n\), we reduce
to
\[
n=(r,0,r,0)^T,
\]
which is projectively equivalent to \(e_0+e_2\).  Thus every null hyperplane is
\(\mathcal{G}_{j,\Qtwotwo}\)-equivalent to
\[
\Pi_0=(e_0+e_2)^{\perp_h}=\{z_0-z_2=0\}.
\]

This proves that the sign of \(\Delta\) is a complete invariant for the \(\mathcal{G}_{j,\Qtwotwo}\)-action on
hyperplanes.
\end{proof}

The proposition refines the direct classification of the previous subsection.  Intrinsically,
the positive and negative cases both yield the same smooth spherical \(CR\)
\(3\)-sphere.  Extrinsically, however, they form two distinct \(\mathcal{G}_{j,\Qtwotwo}\)-orbits.

\begin{corollary}
Let \(\Pi\subset \CP^3\) be a complex hyperplane.
\begin{enumerate}
\item[\rm (i)] If \(\Delta(\Pi)\neq 0\), then \(\Pi\cap \Qtwotwo\) is a smooth spherical
\(CR\) \(3\)-sphere.
\item[\rm (ii)] If \(\Delta(\Pi)=0\), then \(\Pi\) is tangent to \(\Qtwotwo\), and
\(\Pi\cap \Qtwotwo\) is of the singular Levi-flat type.
\item[\rm (iii)] The smooth hyperplane sections split into two distinct \(\mathcal{G}_{j,\Qtwotwo}\)-orbits,
according to the sign of \(\Delta\).
\end{enumerate}
\end{corollary}

\begin{proof}
Statements \((i)\) and \((ii)\) are exactly the content of
Theorem~\ref{thm:first-hyperplane-classification}.  Statement \((iii)\) follows from
Theorem~\ref{thm:hyperplanes-Gj}.
\end{proof}

It is sometimes useful to allow projective transformations preserving $h$ up to sign.
Set
\[
\widehat{G}_j:=\{[A]\in\PGL(4,\C)\mid A^*HA=\pm H,\ AS=S\overline{A}\}.
\]
The matrix
\[
R=\begin{pmatrix}0&0&1&0\\0&0&0&1\\1&0&0&0\\0&1&0&0\end{pmatrix}
\]
satisfies $R^*HR=-H$ and $RS=S\overline{R}$, and sends $e_0$ to $e_2$, so it
interchanges the positive and negative orbits. Under $\widehat{G}_j$ the three
$\mathcal{G}_{j,\Qtwotwo}$-orbits therefore merge into two: $\Delta(\Pi)\neq 0$ and $\Delta(\Pi)=0$.

\begin{remark}
Every non-tangent hyperplane section of $\Qtwotwo$ is intrinsically a copy of the
spherical $CR$ $3$-sphere, but the $\mathcal{G}_{j,\Qtwotwo}$-classification records more, it records on
which side of the split Hermitian geometry the hyperplane lies. The same spherical $CR$
manifold thus appears in two extrinsically distinct positions relative to the twistor
structure, and this distinction collapses only after passing to $\widehat{G}_j$. This three-orbit structure mirrors
the real orbits decomposition of $\CP^3$ into two open $SU(2,2)$-orbits and the unique
closed orbit $\Qtwotwo$, viewed here from the dual perspective of hyperplanes, see
Remark~\ref{rem:orbits}.
\end{remark}

\subsection{Tangent versus non-tangent hyperplanes and their interpretation on the base}
We now reinterpret the hyperplane classification from the twistor point of view. The
key distinction of tangency  corresponds geometrically to whether
the section is a global twistor section over $S^3$ or whether one fibre becomes
exceptional.

Let $M_\Pi:=\Pi\cap\Qtwotwo$. Since every twistor fibre is a projective line, the
geometry of $M_\Pi$ is controlled by whether $\Pi$ contains an entire twistor fibre.

\begin{proposition}
Let $\Pi\subset\CP^3$ be a complex hyperplane.
\begin{enumerate}
\item[\rm(i)] \textbf{Non-tangent case.} $\Pi$ contains no twistor fibre over $S^3$,
every fibre of $\pi|_{\Qtwotwo}$ meets $M_\Pi$ in exactly one point, and
\[
\pi|_{M_\Pi}:M_\Pi\longrightarrow S^3
\]
is a diffeomorphism. The induced $CR$ structure on $S^3$ is the standard spherical one.

\item[\rm(ii)] \textbf{Tangent case.} $\Pi$ is tangent to $\Qtwotwo$ at some
$n\in\Qtwotwo$, and contains the entire twistor fibre $\ell_n$ over $x=\pi(n)\in S^3$.
Every fibre over $S^3\setminus\{x\}$ meets $M_\Pi$ in one point, so
\[
\pi\colon M_\Pi\setminus\ell_n\longrightarrow S^3\setminus\{x\}
\]
is a diffeomorphism, but the section fails to extend globally across $x$.
\end{enumerate}
\end{proposition}

\begin{proof}
In case (i), if $\Pi$ contained a twistor fibre $\ell_x\subset\Qtwotwo$, then
$M_\Pi$ would contain a complex line, forcing the tangent case by
Corollary~\ref{prop:hyperplane-containing-fibre}. Hence no fibre is contained in
$\Pi$, every fibre meets $\Pi$ in exactly one point, and bijectivity together with
compactness gives the diffeomorphism.

In case (ii), write $\Pi=\PP(W)$ with $W=n^{\perp_h}$. Since $h(n,Jn)=0$, one has
$Jn\in W$, so $\ell_n=\PP(\Span_\C\{n,Jn\})\subset\Pi\cap\Qtwotwo$. For
$y\neq x$, the fibre $\pi^{-1}(y)$ is distinct from $\ell_n$ and hence meets $\Pi$
in exactly one point, giving the claimed diffeomorphism on the complement.
\end{proof}

The discriminant locus of the section is the set of points of $S^3$ over which the
fibre is not transverse to $\Pi$. In the non-tangent case it is empty; in the tangent
case it consists of the single point $\{x\}$, which plays the same role as the
discriminant locus in the theory of quadrics developed in
\cite{SalamonViaclovsky2009}, i.e. it is the base-point where the section ceases to be
a graph over the fibration.

The $CR$ geometry of the two cases is strikingly different. In the non-tangent case,
$M_\Pi$ is a smooth spherical $CR$ $3$-sphere, and $\pi|_{M_\Pi}$ transfers the
standard spherical $CR$ structure to $S^3$ , 
recovering the model of
Section~\ref{sec:embedded-spheres}.

In the tangent case the smooth locus of $M_\Pi$ is $\C\times S^1$, which via the
diffeomorphism $S^3\setminus\{x\}\cong\R^3$ equips $\R^3$ with a Levi-flat $CR$
structure. This is intrinsically different from the standard spherical $CR$ structure
on $\R^3\cong S^3\setminus\{\mathrm{pt}\}$ (the Heisenberg group model, which is
Levi-nondegenerate): the tangent section induces a $CR$ structure on the open dense
subset $S^3\setminus\{x\}$ that cannot be extended to a Levi-nondegenerate structure
across the missing point $x$, and is not locally equivalent to the flat model
$\Qtwotwo$ anywhere. This is the simplest instance of the general phenomenon whereby
projective slices of twistor spaces define natural $CR$ structures only on open dense
subsets of the base, rather than globally \cite{SalamonViaclovsky2009}.

\section{\texorpdfstring{$j$}{j}-invariant smooth quadrics and their quadratic sections}\label{sec:j-invariant-quadrics}
We now study smooth $j$-invariant quadrics in $\CP^3$ and their interaction with
$\Qtwotwo$. These are precisely the nondegenerate \emph{real quadrics} of
Salamon--Viaclovsky \cite{SalamonViaclovsky2009} (after translating their quaternionic
notation into the involution $j$ used here), and they form the natural class for
producing branched twistor geometry on the distinguished $3$-sphere $\Sigma\subset S^4$.

\subsection{Matrix description and identification with the quaternionic real form}

Recall from Section~\ref{sec:ambient-models} the anti-holomorphic involution \(j\)
on \(\CP^3\), induced by the anti-linear map \(J(z)=S\overline{z}\) with \(S^TS=I\)
and \(S^2=-I\) \eqref{eq:j-section2}.
A quadric in \(\CP^3\) is given by a nonzero symmetric matrix
\[
Q\in Sym_4(\C),
\qquad
Q^T=Q,
\]
through the equation
\[
\mathcal S_Q
=
\bigl\{
[z]\in \CP^3
\mid
z^TQz=0
\bigr\}.
\]
It is smooth if and only if \(Q\) is nondegenerate, equivalently \(\det Q\neq 0\).

\begin{proposition}
Let \(Q\in Sym_4(\C)\setminus\{0\}\).  Then the following are equivalent:
\begin{enumerate}
\item[\rm(i)] the quadric \(\mathcal S_Q\subset \CP^3\) is \(j\)-invariant;
\item[\rm(ii)] there exists \(\lambda\in\C^\times\) such that
\begin{equation}\label{eq:jinv-up-to-scale-smooth}
S^T\overline Q\,S=\lambda Q.
\end{equation}
\end{enumerate}
Moreover, after multiplying \(Q\) by a nonzero complex scalar, one may normalize
\eqref{eq:jinv-up-to-scale-smooth} to
\begin{equation}\label{eq:jinv-normalized-smooth}
S^T\overline Q\,S=Q.
\end{equation}
Under this normalization, \(Q\) has the form
\begin{equation}\label{eq:jinv-matrix-explicit-smooth}
Q=
\begin{pmatrix}
a& ib& c& d\\
ib& \overline a& -\overline d& \overline c\\
c& -\overline d& e& if\\
d& \overline c& if& \overline e
\end{pmatrix},
\qquad
a,c,d,e\in\C,
\quad
b,f\in\R.
\end{equation}
Hence the normalized \(j\)-invariant symmetric matrices form a real vector space of
dimension \(10\), and the smooth ones are exactly those of the form
\eqref{eq:jinv-matrix-explicit-smooth} with nonzero determinant.
\end{proposition}

\begin{proof}
Set
\[
q(z):=z^TQz.
\]
The quadric \(\mathcal S_Q\) is \(j\)-invariant if and only if
\[
q(z)=0
\quad\Longleftrightarrow\quad
q(Jz)=0
\qquad
\forall\,z\in\C^4\setminus\{0\}.
\]
Since \(Jz=S\overline z\), one computes
\[
q(Jz)
=
(S\overline z)^TQ(S\overline z)
=
\overline z^{\,T}S^TQS\,\overline z
=
z^T\bigl(S^T\overline Q\,S\bigr)z.
\]
Therefore \(\mathcal S_Q\) is \(j\)-invariant if and only if the two nonzero quadratic
forms \(z^TQz\) and \(z^T(S^T\overline Q\,S)z\) define the same projective quadric,
which is equivalent to \eqref{eq:jinv-up-to-scale-smooth}.

Now define
\[
T:Sym_4(\C)\to Sym_4(\C),
\qquad
T(Q):=S^T\overline Q\,S.
\]
Since \(S\) is real and \(S^T S=I\), one has \(T^2=\Id\).  Thus, if \(T(Q)=\lambda Q\),
then
\[
Q=T^2(Q)=T(\lambda Q)=\overline\lambda\,T(Q)=|\lambda|^2Q,
\]
hence \(|\lambda|=1\).  Writing \(\lambda=e^{i\theta}\) and choosing
\(\mu=e^{-i\theta/2}\), one gets
\[
T(\mu Q)=\overline\mu\,T(Q)=\overline\mu\,\lambda Q=\mu Q,
\]
which proves the normalized form \eqref{eq:jinv-normalized-smooth}.

Finally, writing
\[
Q=(q_{rs})_{r,s=1}^4,
\qquad
q_{rs}=q_{sr},
\]
and imposing \(S^T\overline Q\,S=Q\), a direct entrywise computation yields
\[
q_{22}=\overline{q}_{11},
\qquad
q_{12}=-\overline{q}_{12},
\qquad
q_{24}=\overline{q}_{13},
\qquad
q_{23}=-\overline{q}_{14},
\]
\[
q_{44}=\overline{q}_{33},
\qquad
q_{34}=-\overline{q}_{34},
\]
which is exactly \eqref{eq:jinv-matrix-explicit-smooth}.
\end{proof}

It is often convenient to reorder the coordinates as
\[
(z_0,z_2,z_1,z_3),
\]
namely to conjugate by
\[
P=
\begin{pmatrix}
1&0&0&0\\
0&0&1&0\\
0&1&0&0\\
0&0&0&1
\end{pmatrix}.
\]
In this basis the quaternionic structure is represented by
\[
\mathbf J
=
P^{-1}SP
=
\begin{pmatrix}
0&-I_2\\
I_2&0
\end{pmatrix}.
\]
Then every normalized \(j\)-invariant symmetric matrix can be written uniquely in the
block form

\[Q=\{A\mid B\}
:=
\begin{pmatrix}
A&B\\
-\overline B&\overline A
\end{pmatrix},
\qquad
A^T=A,
\qquad
B^*=-B.
\]

\begin{remark}
After the above permutation of basis, the normalized \(j\)-invariant matrices are
exactly the nondegenerate ``real'' symmetric bilinear forms considered by
Salamon--Viaclovsky.  In particular, all the structural results that we use below
apply precisely to the class of quadrics treated in this section.
\end{remark}

\subsection{Classification under the group commuting with \texorpdfstring{$j$}{j}}

Let
\[
\mathcal G_j
:=
\mathbb P\bigl\{G\in GL(4,\C)\mid GS=S\overline G\bigr\}.
\]
This is the projective group of complex linear transformations commuting with the
quaternionic structure; equivalently,
\[
\mathcal G_j\simeq \PGL(2,\HH).
\]
It acts on the space of quadrics by
\[
Q\longmapsto G^TQG.
\]

\begin{theorem}
All smooth $j$-invariant quadrics in $\CP^3$ belong to a single $\mathcal{G}_j$-orbit.
\end{theorem}

\begin{proof}
By \cite[Proposition~4.2]{SalamonViaclovsky2009}, any real nondegenerate symmetric
bilinear form is $GL(2,\HH)$-equivalent to the identity. Projectivizing gives the claim.
\end{proof}

A convenient representative is the Segre-type quadric
\[
\mathcal S_{\mathrm{std}}
=
\bigl\{
[z_0:z_1:z_2:z_3]\in \CP^3
\mid
z_0z_3-z_1z_2=0
\bigr\},
\]
represented by the symmetric matrix
\[Q_{\mathrm{std}}
=
\begin{pmatrix}
0&0&0&1\\
0&0&-1&0\\
0&-1&0&0\\
1&0&0&0
\end{pmatrix}.\]

Thus the absolute classification is trivial: the moduli appear only after fixing
$\Qtwotwo=\pi^{-1}(\Sigma)$.

\subsection{The discriminant circle and its position relative to
\texorpdfstring{$\Sigma$}{Sigma}}
Let $\mathcal{S}\subset\CP^3$ be a smooth $j$-invariant quadric. 
Recall that the twistor projection
\(\pi:\CP^3\longrightarrow S^4\)
has fibres \(\pi^{-1}(p)\cong\CP^1\), the twistor lines. Since \(\mathcal S\) is a quadric, for each
\(p\in S^4\) the restriction of the defining equation of \(\mathcal S\) to the fibre
\(\pi^{-1}(p)\) is a homogeneous polynomial of degree \(2\). Accordingly, the fibre either
meets \(\mathcal S\) in two distinct points, or is tangent to \(\mathcal S\), or is entirely
contained in \(\mathcal S\). The \emph{discriminant locus} is therefore defined as
\[
D(\mathcal S):=
\{p\in S^4\mid \pi^{-1}(p)\cap\mathcal S
\text{ is not two distinct points}\}.
\]
Equivalently, \(D(\mathcal S)\) is the set of points over which the twistor fibre fails to
intersect \(\mathcal S\) transversely in two simple points. One writes
\[
D(\mathcal S)=D_0(\mathcal S)\cup D_1(\mathcal S),
\]
where \(D_0(\mathcal S)\) consists of the points \(p\in S^4\) such that
\(\pi^{-1}(p)\subset\mathcal S\), and \(D_1(\mathcal S)\) consists of the points over which the
intersection is a double point but the fibre is not entirely contained in \(\mathcal S\).
For smooth \(j\)-invariant quadrics, Salamon--Viaclovsky show that
\(D(\mathcal S)=D_0(\mathcal S)\), and that this set is a \emph{geometric circle} in \(S^4\).
Here, by a geometric circle we simply mean a round circle in the sphere, namely the
intersection of \(S^4\subset\R^5\) with a real affine \(2\)-plane in \(\R^5\). In particular,
it need not be a great circle, it is a great circle only when the affine \(2\)-plane passes
through the origin, and in general it is a small circle. Moreover, \(\mathcal S\) is uniquely
determined by this discriminant circle. Thus, once the distinguished round \(3\)-sphere
\(\Sigma\subset S^4\) with \(\Qtwotwo=\pi^{-1}(\Sigma)\) is fixed, the relative geometry of
\(\mathcal S\) is completely encoded by the position of \(D(\mathcal S)\) with respect to
\(\Sigma\)
(see \cite[Theorem~3.10 and Section~3.4]{SalamonViaclovsky2009})

\begin{lemma}\label{lem:relative-position-discriminant-circle}
Let $\Sigma\simeq S^3\subset S^4$ be the distinguished round $3$-sphere and let
$D\subset S^4$ be a geometric circle. Up to a conformal transformation of the pair
$(S^4,\Sigma)$, exactly one of the following four arrangements occurs:
\[
D\subset\Sigma,\qquad
|D\cap\Sigma|=2,\qquad
|D\cap\Sigma|=1,\qquad
D\cap\Sigma=\varnothing.
\]
After a conformal transformation preserving $\Sigma$, one may normalize
\[
S^4=\HH\cup\{\infty\},\qquad
\Sigma=\{q\in\HH:|q|=1\}\cup\{\infty\},\qquad
D\subset\widehat{\C}:=\C\cup\{\infty\}\subset\HH\cup\{\infty\},
\]
so that $D$ becomes a generalized circle in $\widehat{\C}$ and the four cases
above are the four possible relative positions of a generalized circle with respect
to the unit circle $S^1=\Sigma\cap\widehat{\C}$.
\end{lemma}

\begin{proof}
Since $SO(4)\subset\mathrm{Conf}(S^4,\Sigma)$, i.e. the group of conformal transformations of \(S^4\) preserving \(\Sigma\), acts transitively on oriented
$2$-planes in $\R^4$, one may bring $D$ into the complex slice $\widehat{\C}$,
where $\Sigma\cap\widehat{\C}=S^1$. The four cases are the elementary
arrangements of two generalized circles in the Riemann sphere.
\end{proof}

\subsection{Relative classification with respect to
\texorpdfstring{$\Qtwotwo$}{Q22}}

Recall from Section~\ref{sec:symmetries} that
\[\mathcal{G}_{j,\Qtwotwo}\simeq PSp(1,1)\] 
is naturally identified with the conformal
automorphism group of the pair $(S^4,\Sigma)$. Since $\mathcal{S}$ is uniquely
determined by $D(\mathcal{S})$, two
quadrics are $\mathcal{G}_{j,\Qtwotwo}$-equivalent if and only if some conformal
automorphism of $(S^4,\Sigma)$ maps one discriminant circle to the other. The
classification therefore reduces to the conformal geometry of circles in $S^4$
relative to $\Sigma$.

The key invariant for this reduction is the \emph{inversive distance} between two
circles, introduced by Coxeter \cite{Coxeter1966}. After reducing to the planar model
of Lemma~\ref{lem:relative-position-discriminant-circle}, normalize
\[
D(\mathcal{S})=\{|z-a|=r\},
\qquad a\geq 0,\quad r>0,
\]
and
\[
\Sigma\cap\widehat{\C}=S^1=\{|z|=1\}.
\]
For two disjoint circles of radii \(\rho_1,\rho_2\) whose centres are at distance
\(d\), Coxeter's inversive distance \(\delta_{\mathrm{Cox}}\) is defined by first
sending the pair, by a Möbius transformation, to two concentric circles of radii
\(R>r>0\), and then setting
\[
\delta_{\mathrm{Cox}}=\log\frac{R}{r}.
\]
An equivalent formula is
\[
\cosh\delta_{\mathrm{Cox}}
=
\frac{|\rho_1^2+\rho_2^2-d^2|}{2\rho_1\rho_2},
\]
see \cite[formula~(4.2)]{Coxeter1966}. Applying this with \(\rho_1=r\), \(\rho_2=1\),
and \(d=a\) gives
\[
\cosh\delta_{\mathrm{Cox}}
=
\frac{|a^2-r^2-1|}{2r}.
\]

This same algebraic expression continues to make sense even when the circles
intersect. Indeed, if two circles of radii \(\rho_1,\rho_2\) meet at angle
\(\theta\), then at an intersection point the radii to that point are normal to the
circles, so the angle between the circles equals the angle between those radii. By
the cosine rule in the triangle formed by the two centres and an intersection point,
\[
\cos\theta
=
\frac{\rho_1^2+\rho_2^2-d^2}{2\rho_1\rho_2}.
\]
This suggests defining the relative invariant
\begin{equation}\label{eq:inversive-distance}
I\bigl(D(\mathcal{S}),\Sigma\bigr)
:=
\frac{|a^2-r^2-1|}{2r},
\end{equation}
which agrees with \(\cosh\delta_{\mathrm{Cox}}\) in the disjoint case and with
\(|\cos\theta|\) in the intersecting case.

In the present normalization, this quantity completely encodes the relative position
of \(D(\mathcal S)\) and \(S^1\). Indeed,
\[
I<1
\iff
|a^2-r^2-1|<2r
\iff
(r-1)^2<a^2<(r+1)^2
\iff
|r-1|<a<r+1,
\]
which is exactly the condition for two circles of radii \(r\) and \(1\) with centre
distance \(a\) to meet in two distinct points. Likewise,
\[
I=1
\iff
|a^2-r^2-1|=2r
\iff
a=r+1 \ \text{or}\ a=|r-1|,
\]
which is precisely the tangency condition, while
\[
I>1
\iff
a>r+1 \ \text{or}\ a<|r-1|,
\]
which means that the circles are disjoint. Finally, \(I=0\) if and only if
\(a^2=r^2+1\), i.e.\ the circles meet orthogonally. Thus
\begin{equation}\label{eq:I-cases}
I<1\Longleftrightarrow|D\cap\Sigma|=2,\qquad
I=1\Longleftrightarrow|D\cap\Sigma|=1,\qquad
I>1\Longleftrightarrow D\cap\Sigma=\varnothing.
\end{equation}

\begin{remark}
The quantity \(I(D(\mathcal{S}),\Sigma)\) is well defined as a relative invariant
because the inversive distance is preserved under Möbius transformations, two pairs
of circles can be mapped to one another by a Möbius transformation if and only if
they have the same inversive distance \cite{Coxeter1966}. In particular, \(I\) is
independent of the choice of planar normal form in
Lemma~\ref{lem:relative-position-discriminant-circle}.

Moreover, in the present setting \(I\) is a complete invariant for the relative
position of \(D(\mathcal S)\) with respect to \(\Sigma\): after reduction to the
planar model, two pairs \((D(\mathcal S),\Sigma)\) and \((D(\mathcal S'),\Sigma)\)
with the same value of \(I\) are Möbius-equivalent by a transformation preserving
\(S^1\), hence by a conformal transformation of \((S^4,\Sigma)\). Combined with the
fact that a smooth \(j\)-invariant quadric is uniquely determined by its discriminant
circle, this shows that two such quadrics with the same inversive distance from
\(\Sigma\) are \(\mathcal{G}_{j,\Qtwotwo}\)-equivalent. 
\end{remark}

\begin{theorem}
Let $\mathcal{S},\mathcal{S}'\subset\CP^3$ be smooth $j$-invariant quadrics. The
following are equivalent:
\begin{enumerate}
\item[\rm(i)] $\mathcal{S}$ and $\mathcal{S}'$ are $\mathcal{G}_{j,\Qtwotwo}$-equivalent;
\item[\rm(ii)] the pairs $(D(\mathcal{S}),\Sigma)$ and $(D(\mathcal{S}'),\Sigma)$ are
conformally equivalent in $S^4$.
\end{enumerate}
Consequently:
\begin{enumerate}
\item[\rm(a)] the quadrics with $D(\mathcal{S})\subset\Sigma$ form a single
distinguished orbit;
\item[\rm(b)] if $D(\mathcal{S})\not\subset\Sigma$, the
$\mathcal{G}_{j,\Qtwotwo}$-orbit is parametrized by
$I(D(\mathcal{S}),\Sigma)\in[0,+\infty)$.
\end{enumerate}
\end{theorem}

\begin{proof}
The equivalence (i)$\Leftrightarrow$(ii) follows from the uniqueness of the quadric
given its discriminant circle. For case (a), the conformal group of $\Sigma\simeq S^3$
acts transitively on circles in $\Sigma$. For case (b),
Lemma~\ref{lem:relative-position-discriminant-circle} reduces the problem to the
classification of generalized circles in $\widehat{\C}$ relative to $S^1$ under
$PSU(1,1)$, which is exactly the one-parameter classification by $I$.
\end{proof}

\begin{corollary}\label{cor:branch-locus-types-s3}
Let $\mathcal{S}\subset\CP^3$ be a smooth $j$-invariant quadric and set
\[\Gamma_{\mathcal{S}}:=D(\mathcal{S})\cap\Sigma.\] Exactly four possibilities occur,
characterized by \eqref{eq:I-cases}:
\begin{enumerate}
\item[\rm(i)] $\Gamma_{\mathcal{S}}$ is a circle $\Longleftrightarrow$
$D(\mathcal{S})\subset\Sigma$;
\item[\rm(ii)] $|\Gamma_{\mathcal{S}}|=2\Longleftrightarrow$
$D(\mathcal{S})\not\subset\Sigma$ and $I<1$;
\item[\rm(iii)] $|\Gamma_{\mathcal{S}}|=1\Longleftrightarrow$
$D(\mathcal{S})\not\subset\Sigma$ and $I=1$;
\item[\rm(iv)] $\Gamma_{\mathcal{S}}=\varnothing\Longleftrightarrow$
$D(\mathcal{S})\not\subset\Sigma$ and $I>1$.
\end{enumerate}
\end{corollary}

\begin{proof}
After reduction to the planar model, this is the elementary classification of
the relative position of $\{|z-a|=r\}$ and $\{|z|=1\}$, encoded by \eqref{eq:I-cases}.
\end{proof}
We now record an explicit family of smooth $j$-invariant quadrics that realizes all
four cases of Corollary~\ref{cor:branch-locus-types-s3}.

\begin{proposition}\label{prop:explicit-family-smooth-j-invariant-quadrics}
For $a\in\R$ and $r>0$, let
\[\mathcal{S}_{a,r}=\bigl\{z_0z_1-a(z_0z_3+z_1z_2)+(a^2-r^2)z_2z_3=0\bigr\}
\subset\CP^3.\]
Then:
\begin{enumerate}
\item[\rm(i)] $\mathcal{S}_{a,r}$ is smooth and projectively $j$-invariant;
\item[\rm(ii)] it is represented by the symmetric matrix
\begin{equation}\label{eq:Qa-r-matrix}
Q_{a,r}=\begin{pmatrix}
0&1&0&-a\\1&0&-a&0\\0&-a&0&a^2-r^2\\-a&0&a^2-r^2&0
\end{pmatrix},
\end{equation}
with $S^T\overline{Q_{a,r}}S=-Q_{a,r}$ and $\det Q_{a,r}=r^4$;
\item[\rm(iii)] the discriminant circle, in the affine chart
$S^4\setminus\{\infty\}\simeq\HH$ with $q=p_0+jp_1$, is
\begin{equation}\label{eq:Da-r-discriminant}
D_{a,r}=\bigl\{p_1=0,\quad 1-2a\,\Re(p_0)+(a^2-r^2)|p_0|^2=0\bigr\}.
\end{equation}
\end{enumerate}
This family realizes all four cases of Corollary~\ref{cor:branch-locus-types-s3}:
\begin{enumerate}
\item[\rm(a)] $(a,r)=(0,1)$: $D_{0,1}=\{p_1=0,|p_0|=1\}=S^1\subset\Sigma$
\emph{(contained)};
\item[\rm(b)] $(a,r)=(1,1)$: $D_{1,1}=\{p_1=0,\Re(p_0)=\tfrac{1}{2}\}\cup\{\infty\}$,
meeting $\Sigma$ in two points, $I=\tfrac{1}{2}<1$ \emph{(transverse)};
\item[\rm(c)] $(a,r)=(2,1)$: $D_{2,1}=\{p_1=0,|p_0-\tfrac{2}{3}|=\tfrac{1}{3}\}$,
tangent to $S^1$ at $p_0=1$, $I=1$ \emph{(tangent)};
\item[\rm(d)] $(a,r)=(0,\tfrac{1}{2})$: $D_{0,1/2}=\{p_1=0,|p_0|=2\}$,
disjoint from $S^1$, $I=\tfrac{5}{4}>1$ \emph{(disjoint)}.
\end{enumerate}
\end{proposition}

\begin{proof}
Let
\[
F_{a,r}(z)=z_0z_1-a(z_0z_3+z_1z_2)+(a^2-r^2)z_2z_3
\]
and put \(c=a^2-r^2\). With the convention that the quadratic
form is represented by \(z^TQ_{a,r}z\), the matrix
\eqref{eq:Qa-r-matrix} represents \(2F_{a,r}\). Hence it defines the
same quadric. A direct computation gives
\[
\det Q_{a,r}=r^4.
\]
Since \(r>0\), the matrix has full rank, and therefore
\(\mathcal S_{a,r}\) is a smooth quadric in \(\CP^3\).

For \(j\)-invariance, recall that
$S=\left(\begin{smallmatrix}0&-1&0&0\\1&0&0&0\\0&0&0&-1\\0&0&1&0
\end{smallmatrix}\right).$
Since \(Q_{a,r}\) is real, a direct computation gives
\[
S^T\overline{Q_{a,r}}S=S^TQ_{a,r}S=-Q_{a,r}.
\]
Equivalently,
\[
F_{a,r}(S\overline z)=-\overline{F_{a,r}(z)}.
\]
Thus the zero locus of \(F_{a,r}\) is preserved by the projective
antiholomorphic involution induced by \(S\). This proves projective
\(j\)-invariance.

We now compute the discriminant locus. Write the twistor fibre over
\(q=p_0+jp_1\in\HH\) as
\[
\begin{cases}z_2=z_0p_0-z_1\overline{p}_1,\\
z_3=z_0p_1+z_1\overline{p}_0.
\end{cases}
\]
Substituting into \(F_{a,r}\) gives a binary quadratic form
\[
F_{a,r}|_{\ell_p}
=
\alpha(p)z_0^2+2\beta(p)z_0z_1+\gamma(p)z_1^2,
\]
where
\[
\alpha(p)=\bigl((a^2-r^2)p_0-a\bigr)p_1,\qquad
\gamma(p)=\bigl(a-(a^2-r^2)\overline p_0\bigr)\overline p_1,
\]
and
\[
2\beta(p)
=
1-a(p_0+\overline p_0)
+(a^2-r^2)\bigl(|p_0|^2-|p_1|^2\bigr).
\]
The discriminant is
\[
\Delta(p)=\beta(p)^2-\alpha(p)\gamma(p).
\]
Since \(a,r\in\R\), we obtain
\[
\Delta(p)
=
\frac14
\left(
1-2a\,\Re(p_0)
+(a^2-r^2)(|p_0|^2-|p_1|^2)
\right)^2
+
|p_1|^2
\left|
(a^2-r^2)p_0-a
\right|^2.
\]
This is a sum of non-negative real terms.

Hence \(\Delta(p)=0\) implies
\[
1-2a\,\Re(p_0)
+(a^2-r^2)(|p_0|^2-|p_1|^2)=0
\]
and either \(p_1=0\) or \((a^2-r^2)p_0=a\). The second possibility is
incompatible with the first equation when \(p_1\neq 0\): indeed, if
\(c=a^2-r^2\neq0\) and \(cp_0=a\), then 
by substituting \(p_0=a/c\in\R\)
in the first equation we get
\[
\Delta(p)=1-2\frac{a^2}{c}+c\left(\frac{a^2}{c^2}-|p_1|^2 \right)= 1-\frac{a^2}{c}-c|p_1|^2=-\frac{r^2}{c}-c|p_1|^2=-\frac{r^2+c^2|p_1|^2}{c}=0,
\]
which is impossible, since the two summands have the same sign. If
\(c=0\), then \((a^2-r^2)p_0=a\) would imply \(a=0\), impossible because
\(r>0\) and \(a^2=r^2\).

Therefore the discriminant locus is precisely given by \(p_1=0\) and
\[
1-2a\,\Re(p_0)+(a^2-r^2)|p_0|^2=0.
\]
This proves \eqref{eq:Da-r-discriminant}.

It remains to identify the four cases. In the affine coordinate
\(z=p_0\) on the slice \(\{p_1=0\}\simeq\C\), the equation of
\(D_{a,r}\) is
\(1-2a\,\Re z+(a^2-r^2)|z|^2=0\).
Put \(c=a^2-r^2\). If \(c\neq 0\), dividing by \(c\) and completing the square gives
\[
|z|^2-\frac{2a}{c}\Re z+\frac{1}{c}=0,
\]
hence
\[
\left|z-\frac{a}{c}\right|^2
=
\frac{a^2}{c^2}-\frac{1}{c}
=
\frac{a^2-c}{c^2}
=
\frac{r^2}{c^2}.
\]
Therefore
\[
D_{a,r}
=
\left\{
\left|z-\frac{a}{a^2-r^2}\right|
=
\frac{r}{|a^2-r^2|}
\right\}.
\]
If \(a^2\neq r^2\), this is the circle
\[
\left|z-\frac{a}{a^2-r^2}\right|
=
\frac{r}{|a^2-r^2|}.
\]
If \(a^2=r^2\), it is the affine line
\[
\Re z=\frac{1}{2a},
\]
whose compactification in \(\widehat{\C}\) is a circle through
\(\infty\).

For \((a,r)=(0,1)\), this gives
\[
D_{0,1}=\{|z|=1\}=S^1\subset\Sigma.
\]
For \((a,r)=(1,1)\), we get
\[
D_{1,1}=\{\Re z=\tfrac12\}\cup\{\infty\},
\]
which meets \(S^1\) transversely in two points. For \((a,r)=(2,1)\),
we obtain
\[
D_{2,1}=\left\{\left|z-\frac23\right|=\frac13\right\},
\]
which is tangent to \(S^1\) at \(z=1\). Finally, for
\((a,r)=(0,\tfrac12)\), we obtain
\[
D_{0,1/2}=\{|z|=2\},
\]
which is disjoint from \(S^1\).

For the inversive distances, note that the circle \(D_{a,r}\) is the
image under inversion \(z=1/w\) of the Euclidean circle
\[
C_{a,r}=\{|w-a|=r\}.
\]
Indeed,
\[
\left|\frac1z-a\right|^2=r^2
\]
is equivalent, after multiplying by \(|z|^2\), to
\[
1-2a\,\Re z+(a^2-r^2)|z|^2=0.
\]
Since inversion preserves \(S^1\) and preserves inversive distance, we
may compute \(I\) using \(C_{a,r}\). Thus
\[
I=\frac{|a^2-r^2-1|}{2r}.
\]
Therefore
\[
(a,r)=(1,1)\colon\quad I=\frac{|1-1-1|}{2}=\frac12<1,
\]
\[
(a,r)=(2,1)\colon\quad I=\frac{|4-1-1|}{2}=1,
\]
and
\[
(a,r)=\left(0,\frac12\right)\colon\quad
I=\frac{\left|0-\frac14-1\right|}{1}=\frac54>1.
\]
This gives respectively the transverse, tangent, and disjoint cases.
\end{proof}

A picture of the four behaviours described in Proposition~\ref{prop:explicit-family-smooth-j-invariant-quadrics} is contained in Figure~\ref{fig:four-circle-positions}.

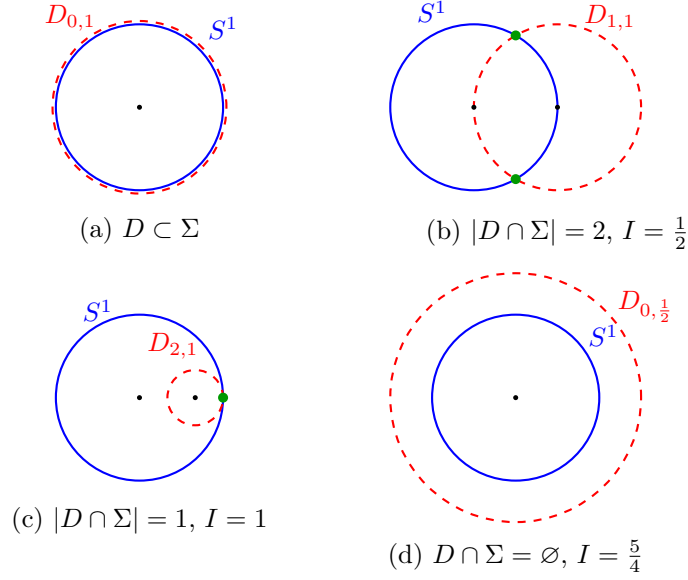
\begin{figure}[h]
\centering
\begin{tikzpicture}[scale=1.1]

\begin{scope}[xshift=0cm, yshift=3.5cm]
  \draw[thick, blue] (0,0) circle (1);
  \draw[thick, red, dashed] (0,0) circle (1.04);
  \fill (0,0) circle (0.03);
  \node[below] at (0,-1.2) {\small (a) $D\subset\Sigma$};
  \node[above right, blue, font=\small] at (0.7,0.7) {$S^1$};
  \node[above left, red, font=\small] at (-0.4,0.8) {$D_{0,1}$};
\end{scope}

\begin{scope}[xshift=4.5cm, yshift=3.5cm]
  \draw[thick, blue] (-0.5,0) circle (1);
  \draw[thick, red, dashed] (0.5,0) circle (1);
  \fill[green!60!black] (0, 0.866) circle (0.06);
  \fill[green!60!black] (0,-0.866) circle (0.06);
  \fill (-0.5,0) circle (0.03);
  \fill (0.5,0) circle (0.03);
  \node[below] at (0.5,-1.2) {\small (b) $|D\cap\Sigma|=2$, $I=\tfrac{1}{2}$};
  \node[above right, blue, font=\small] at (-1.3,0.9) {$S^1$};
  \node[above right, red, font=\small] at (0.7,0.8) {$D_{1,1}$};
\end{scope}

\begin{scope}[xshift=0cm, yshift=0cm]
  \draw[thick, blue] (0,0) circle (1);
  \draw[thick, red, dashed] (0.667,0) circle (0.333);
  \fill[green!60!black] (1,0) circle (0.06);
  \fill (0,0) circle (0.03);
  \fill (0.667,0) circle (0.03);
  \node[below] at (0,-1.2) {\small (c) $|D\cap\Sigma|=1$, $I=1$};
  \node[above right, blue, font=\small] at (-0.8,0.8) {$S^1$};
  \node[above, red, font=\small] at (0.4,0.35) {$D_{2,1}$};
\end{scope}

\begin{scope}[xshift=4.5cm, yshift=0cm]
  \draw[thick, blue] (0,0) circle (1);
  \draw[thick, red, dashed] (0,0) circle (1.5);
  \fill (0,0) circle (0.03);
  \node[below] at (0,-1.6) {\small (d) $D\cap\Sigma=\varnothing$, $I=\tfrac{5}{4}$};
  \node[right, blue, font=\small] at (0.75,0.75) {$S^1$};
  \node[right, red, font=\small] at (1.1,1.1) {$D_{0,\frac{1}{2}}$};
\end{scope}

\end{tikzpicture}
\caption{The four relative positions of the discriminant circle $D_{a,r}$ (red, dashed) with
respect to $\Sigma\cap\widehat{\C}=S^1$ (blue) in the planar model, for the explicit
family of Proposition~\ref{prop:explicit-family-smooth-j-invariant-quadrics}.
Green dots mark the branch points $\Gamma_{\mathcal{S}}=D_{a,r}\cap S^1$.
The inversive distance $I$ of \eqref{eq:inversive-distance} distinguishes the four
cases.}
\label{fig:four-circle-positions}
\end{figure}

Let $M_{\mathcal{S}}:=\mathcal{S}\cap\Qtwotwo$. Over $\Sigma\setminus\Gamma_{\mathcal{S}}$
each twistor fibre meets $\mathcal{S}$ in exactly two distinct points, so $\pi_{|M_{\mathcal{S}}\cap\pi^{-1}(\Sigma\setminus\Gamma_{\mathcal{S}})}$
is a two-sheeted unbranched covering over \(\Sigma\setminus\Gamma_{\mathcal{S}}\).
\begin{remark}
The explicit formulas of Section~\ref{sec:explicit-CR} show that on the smooth
locus of \(M_{\mathcal S}\), the induced \(CR\) bundle is simply
\[
T^{0,1}M_{\mathcal S}=T^{0,1}\Qtwotwo\cap \C TM_{\mathcal S},
\]
and  the Levi form of \(M_{\mathcal S}\) is studied relative to the fixed
split Levi form of \(\Qtwotwo\).
\end{remark}
\begin{proposition}\label{prop:local-branches-cr}
Let $U\subset\Sigma\setminus\Gamma_{\mathcal{S}}$ be simply connected. Then
$M_{\mathcal{S}}\cap\pi^{-1}(U)$ splits into two real-analytic sections
$s_\pm:U\to\mathcal{S}$ with $\pi\circ s_\pm=\Id_U$ and $s_-=j\circ s_+$. Each
image $s_\pm(U)\subset\mathcal{S}$ is a real-analytic hypersurface in the complex
surface $\mathcal{S}$, and induces on $U$ a real-analytic $CR$ structure of
hypersurface type
\[
H^{0,1}_\pm U:=s_\pm^{*}\bigl(T^{0,1}\mathcal{S}\cap\C T(s_\pm(U))\bigr).
\]
The two induced $CR$ structures are conjugate.
\end{proposition}

\begin{proof}
By the preceding discussion, the restriction
\[
\pi:M_{\mathcal{S}}\cap\pi^{-1}(U)\longrightarrow U
\]
is a two-sheeted unbranched covering. Since $U$ is simply connected, this covering
is trivial. Hence there exist two global sections
\[
s_\pm:U\longrightarrow M_{\mathcal{S}}\cap\pi^{-1}(U)\subset\mathcal{S},
\qquad
\pi\circ s_\pm=\Id_U,
\]
unique up to permutation. Because $\pi$ and $M_{\mathcal{S}}$ are real-analytic,
the local inverse branches are real-analytic, and therefore so are $s_\pm$.

Since $\mathcal{S}$ is $j$-invariant and $j$ preserves each twistor fibre, $j$
maps the two-point set $\mathcal{S}\cap\pi^{-1}(x)$ to itself for every $x\in U$.
As $j$ has no fixed points on a twistor fibre, it exchanges the two points. After
relabeling if necessary, one obtains
\[
s_-=j\circ s_+.
\]

Each $s_\pm$ is an immersion, because $\pi\circ s_\pm=\Id_U$. Thus $s_\pm(U)$ is
a real-analytic embedded submanifold of $\mathcal{S}$ of real dimension $3$.
Since $\dim_\C\mathcal{S}=2$, this is a real hypersurface in the complex surface
$\mathcal{S}$. It therefore carries its standard induced hypersurface $CR$
structure, and pulling it back by $s_\pm$ gives the bundles $H^{0,1}_\pm U$ above.

Finally, $j$ is anti-holomorphic on $\CP^3$, hence also on $\mathcal{S}$, so its
differential interchanges $T^{1,0}\mathcal{S}$ and $T^{0,1}\mathcal{S}$.
Therefore the two induced $CR$ structures are conjugate.
\end{proof}

\begin{remark} The Levi-nondegeneracy cannot be asserted in the generality of
Proposition~\ref{prop:local-branches-cr}. The Levi-nondegeneracy of the induced
hypersurface \(CR\) structure is not automatic, i.e. transversality to the twistor
fibres does not exclude the 
degeneracy of the Levi form.

Indeed, consider the explicit quadric $\mathcal{S}_{2,\sqrt3}$, for which
$a^2-r^2=1$. In the affine chart $U_0=\{z_0\neq0\}$, with coordinates
$u_j=z_j/z_0$, its equation is
\[
u_1(1-2u_2)=u_3(2-u_2).
\]
Near the point
\[
p=[1:1:-1:1]\in M_{\mathcal{S}_{2,\sqrt3}},
\]
which lies over the base point $q=j\in\Sigma\setminus\Gamma_{\mathcal{S}_{2,\sqrt3}}$,
one may solve for $u_1$ as
\[
u_1=u_3\,\frac{2-u_2}{1-2u_2}.
\]
Hence $M_{\mathcal{S}_{2,\sqrt3}}\subset\mathcal{S}_{2,\sqrt3}$ is locally defined by
\[
\rho(u_2,u_3)
=
1+\left|\frac{2-u_2}{1-2u_2}\right|^2|u_3|^2-|u_2|^2-|u_3|^2=0.
\]
At the point $(u_2,u_3)=(-1,1)$ one has
\[
\rho=0,\qquad
\frac{\partial\rho}{\partial u_2}=\frac{4}{3},\qquad
\frac{\partial\rho}{\partial u_3}=0.
\]
Therefore $T^{1,0}_pM_{\mathcal{S}_{2,\sqrt3}}$ is spanned by
$\partial/\partial u_3$, and the Levi form vanishes on this nonzero complex tangent
vector, because
\[
\frac{\partial^2\rho}{\partial u_3\,\partial\overline{u}_3}(p)
=
\left(\left|\frac{2-u_2}{1-2u_2}\right|^2-1\right)\Bigg|_{u_2=-1}
=0.
\]
So the induced $CR$ structure is Levi-degenerate at $p$. 
\end{remark}

\begin{theorem}

Let \(\mathcal S\subset\CP^3\) be a smooth \(j\)-invariant quadric, and set
\[
M_{\mathcal S}:=\mathcal S\cap Q^{2,2}.
\]
Let \(p\in M_{\mathcal S}\) be a point at which the intersection is transverse.
Then \(M_{\mathcal S}\) is a smooth \(CR\) hypersurface in the complex surface
\(\mathcal S\), and its \(CR\) structure is given by
\[
T^{1,0}_pM_{\mathcal S}
=
T^{1,0}_p\mathcal S\cap T^{1,0}_pQ^{2,2}.
\]
Moreover, up to the usual sign convention, the Levi form of \(M_{\mathcal S}\) at
\(p\) is the restriction of the ambient Levi form of \(Q^{2,2}\) to this complex
distribution. In particular, \(M_{\mathcal S}\) is Levi-nondegenerate at \(p\) if and only
if the line \(T^{1,0}_pM_{\mathcal S}\) is non-isotropic with respect to the 
Levi form of \(Q^{2,2}\).
\end{theorem}

\begin{proof}
Choose a local defining function \(\rho\) for \(Q^{2,2}\) near \(p\), so that
\(Q^{2,2}=\{\rho=0\}\) and \(d\rho\neq0\) along \(Q^{2,2}\). Since \(\mathcal S\)
is a complex surface and the intersection is transverse, the restriction
\[
\rho_{\mathcal S}:=\rho|_{\mathcal S}
\]
is a defining function for \(M_{\mathcal S}\) as a real hypersurface in
\(\mathcal S\). Hence
\[
T^{1,0}_pM_{\mathcal S}
=
\ker(\partial\rho_{\mathcal S})_p
\subset T^{1,0}_p\mathcal S.
\]
Because \(\rho_{\mathcal S}\) is the restriction of \(\rho\), this is exactly
\[
T^{1,0}_p\mathcal S\cap T^{1,0}_pQ^{2,2}.
\]

Now let \(v\in T^{1,0}_pM_{\mathcal S}\). Up to sign, the Levi form of the
hypersurface \(M_{\mathcal S}\subset\mathcal S\) is given by
\[
\mathcal L^{M_{\mathcal S}}_p(v,\bar v)
=
i\,\partial\bar\partial(\rho_{\mathcal S})_p(v,\bar v).
\]
Since \(\mathcal S\) is a complex submanifold, restricting \(\rho\) to
\(\mathcal S\) does not change the \((1,1)\)-form \(i\,\partial\bar\partial\rho\)
on vectors tangent to \(\mathcal S\). Therefore
\[
\mathcal L^{M_{\mathcal S}}_p(v,\bar v)
=
\mathcal L^{Q^{2,2}}_p(v,\bar v)
\]
up to sign, where on the right-hand side \(v\) is viewed as an element of
\(T^{1,0}_pQ^{2,2}\). This proves the claim.
\end{proof}

\begin{corollary}
In the setting of Proposition~\ref{prop:local-branches-cr}, let
\(s_\pm:U\to\mathcal S\) be one of the local branches. Then the induced
\(CR\) structure on \(U\) is Levi-nondegenerate at \(x\in U\) if and only if the
complex line
\[
T^{1,0}_{s_\pm(x)}\mathcal S\cap T^{1,0}_{s_\pm(x)}Q^{2,2}
\]
is non-isotropic with respect to the  Levi form of \(Q^{2,2}\).
\end{corollary}

\begin{corollary}\label{cor:global-cr-cases}
The branch locus \(\Gamma_{\mathcal S}\) determines the global domain of the
induced \(CR\) geometry as follows:
\begin{enumerate}
\item[\rm(i)] If \(\Gamma_{\mathcal S}=\varnothing\), then, since
\(\Sigma\simeq S^3\) is simply connected, the two-sheeted covering splits
globally, and one obtains two globally defined conjugate real-analytic
\(CR\) structures on \(S^3\).

\item[\rm(ii)] If \(\Gamma_{\mathcal S}=\{p\}\), then one obtains two globally
defined conjugate real-analytic \(CR\) structures on
\[
\Sigma\setminus\{p\}\cong\R^3.
\]

\item[\rm(iii)] If \(\Gamma_{\mathcal S}=\{p,q\}\), then one obtains two globally
defined conjugate real-analytic \(CR\) structures on
\[
\Sigma\setminus\{p,q\}\cong \R^3\setminus\{0\}\cong S^2\times\R
\cong \CP^1\times(0,+\infty).
\]

\item[\rm(iv)] If \(\Gamma_{\mathcal S}\) is a circle, then the induced
\(CR\) structures are defined locally on the two branches of the covering over
\(\Sigma\setminus\Gamma_{\mathcal S}\), but there is no canonical global
splitting.
\end{enumerate}
\end{corollary}

\begin{proof}
Cases \rm(i)--\rm(iii) follow from
Proposition~\ref{prop:local-branches-cr}, since the manifolds
\[
S^3,\qquad S^3\setminus\{p\}\cong\R^3,\qquad
S^3\setminus\{p,q\}\cong S^2\times\R
\]
are simply connected, and therefore the two-sheeted covering splits globally.
Case \rm(iv) is intrinsically only local, since the covering over
\(\Sigma\setminus\Gamma_{\mathcal S}\) does not come equipped with a preferred
global trivialization.
\end{proof}

\begin{remark}
The \(CR\) structures arising in case \rm(i) of
Corollary~\ref{cor:global-cr-cases} deserve particular attention. Each section
\(s_\pm:S^3\to\mathcal S\) realizes \(S^3\) as a real-analytic hypersurface in
the complex surface \(\mathcal S\cong\CP^1\times\CP^1\), and the induced
\(CR\) structure is therefore the hypersurface \(CR\) structure determined by
this embedding, rather than the flat \(CR\) structure of
\(Q^{2,2}\subset\CP^3\). The relative position of the discriminant circle, equivalently the parameter
\(I>1\), produces a natural one-parameter family of such embeddings. It would
be interesting to determine for which values of \(I\) the resulting
\(CR\) structures are spherical, everywhere Levi-nondegenerate, or mutually
\(CR\)-equivalent. These intrinsic questions appear subtler than the ambient
\(\mathcal G_{j,Q^{2,2}}\)-classification, and we do not pursued them here.
\end{remark}
\begin{remark}
    In the case \(\Gamma_{\mathcal S}=\varnothing\), the construction yields a
one-parameter family of real-analytic \(CR\) structures on the smooth manifold
\(S^3\). At present, however, we do not identify these structures intrinsically
with Cartan's homogeneous models. Such an identification would require, in
addition, either a computation of the induced Cartan curvature, or a proof that
the resulting \(CR\) structures are homogeneous together with a comparison of
their structure equations with Cartan's list. In the latter direction, Porter's
classification of homogeneous \(CR\) \(3\)-folds \cite{Porter2021} equivariantly embedded in the
  hyperquadric shows that, among Levi-nondegenerate \(\mathfrak{su}(2)\)-models,
at most a special member can occur in \(Q^{2,2}\).
\end{remark}
\section*{Acknowledgements}
The authors were partially supported by PRIN 2022MWPMAB --- ``Interactions between Geometric Structures and Function Theories'', CUP: H53D23002040006, and by the group G.N.S.A.G.A. of I.N.d.A.M.\ . Moreover, the second author was supported by University of Parma through the action ``Bando di Ateneo 2025 per la ricerca''

\end{document}